\documentclass{amsart}
\usepackage{enumerate,enumitem}
\usepackage{xcolor}
\usepackage{amssymb}
\usepackage{geometry}
\usepackage{hyperref}

\hypersetup{
colorlinks=true,
linkcolor=blue,
citecolor=red,
urlcolor=olive,
}
\newtheorem{theorem}{Theorem}[section]
\newtheorem{lemma}[theorem]{Lemma}
\newtheorem{proposition}[theorem]{Proposition}

\theoremstyle{definition}
\newtheorem{definition}[theorem]{Definition}
\newtheorem{remark}[theorem]{Remark}

\makeatletter
\@namedef{subjclassname@2020}{  \textup{2020} Mathematics Subject Classification}
\makeatother
\numberwithin{equation}{section}

\begin{document}
\author{M. C. Bortolan}
\address[M. C. Bortolan]{Departamento de Matem\'atica, Universidade Federal de Santa
Catarina, Florian\'opolis - SC, Brasil.}
\email{m.bortolan@ufsc.br}
\author{A. N. Carvalho}
\address[A. N. Carvalho]{Instituto de Ci\^{e}ncias Ma\-te\-m\'{a}\-ti\-cas e de Computa\c{c}\~{a}o\\ 
Universidade de S\~{a}o Paulo, Campus de S\~{a}o Carlos, Caixa Postal 668, S\~{a}o Carlos SP, Brazil.}
\thanks{(A. N. Carvalho) Partially supported by FAPESP grants 2020/14075-6 and CNPq 306213/2019-2}
\email{andcarva@icmc.usp.br}
\author{P. Mar\'in-Rubio}
\address[P. Mar\'in-Rubio]{Departamento de Ecuaciones Diferenciales y An\'alisis
Num\'erico, Universidad de Sevilla, C/Tarfia s/n, 41012-Sevilla, Espa\~na.}
\thanks{(P. Mar\'in-Rubio) Partially supported by the Spanish Ministry of Science, Innovation and
Universities, project PGC2018-096540-B-I00, by the Junta de Andaluc\'{\i}a
and FEDER, projects US-1254251 and P18-FR-4509.}
\email{pmr@us.es}
\author{J. Valero}
\address[J. Valero]{Departamento de Estad\'istica, Matem\'aticas e Inform\'atica,
Universidad Miguel Hernandez de Elche, Elche - Alicante, Espa\~na.}
\thanks{(J. Valero) Partially supported by the Spanish Ministry of Science, Innovation and Universities, project PGC2018-096540-B-I00, by the Spanish Ministry of Science and Innovation, project PID2019-108654GB-I00, and by the Junta de Andalucía and FEDER, project P18-FR-4509.}
\email{jvalero@umh.es}
\title[Weak global attractors]{Weak global attractor for the $3D$%
-Navier-Stokes equations via the globally modified Navier-Stokes equations}
\keywords{$3D$-Navier-Stokes equations, weak global attractors, globally
modified Navier-Stokes equations, semilinear parabolic equations, $\epsilon$%
-regularity}
\subjclass[2020]{ 35Q30, 35B41, 35K58, 76D05. }

\begin{abstract}
In this paper we obtain the existence of a \textit{weak global attractor}
for the three-dimensional Navier-Stokes equations, that is, a weakly compact
set with an invariance property, that uniformly attracts solutions, with
respect to the weak topology, for initial data in bounded sets. To that end,
we define this weak global attractor in terms of limits of solutions of the
globally modified Navier-Stokes equations in the weak topology. We use the
theory of semilinear parabolic equations and $\epsilon$-regularity to obtain
the local well posedness for the globally modified Navier-Stokes equations
and the existence of a global attractor and its regularity.
\end{abstract}

\maketitle

\section{Introduction}

The understanding of the asymptotic behavior of solutions for the
three-dimensional Navier-Stokes equations is still a big challenge. In
particular, the existence of the global attractor with respect to the strong
topology of the phase space is up to now an open problem. The main
difficulties in proving such results come from the facts that it is not
known whether the weak solutions are continuous and also if uniqueness of
the Cauchy problem is true. As proved in \cite{Ball00} as a conditional
result, if all the weak solution were continuous, then they would define a
multivalued semiflow possessing a global attractor with respect to the
strong topology (an extension of this to the stochastic framework is \cite%
{mrr}). Later on, in \cite{KapVal07} this result was extended by showing
that if strong solutions were globally defined, then at least one weak
solution would be continuous, ensuring then the existence of a multivalued
semiflow having a global attractor, which is again a conditional theorem.

Another approach that has been helpful in order to handle this problem is to
use the theory of trajectory attractors (see \cite{CheVi97,CheTitiVi07,Sell,CheVi02,CheTitiVi07B} for
instance, also \cite{CheskidovLu} for the nonautonomous case, and \cite%
{FlandoliSchmalfuss}\ for the stochastic equation). The idea behind these
types of attractors is to avoid the problem of the lack of continuity by
using a weaker topology, namely, the topology of square integrable functions
on finite intervals of time. The drawback of this kind of attractor is that
the connection with the original physical phase space is lost.

The existence of the global attractor has been obtained when we
consider the weak topology of the phase space. In view of the lack of
uniqueness, the proper way to do that is to define a multivalued semiflow
and prove at least the existence and negative invariance of the attractor.
This question is fraught with important difficulties. If we use the
Leray-Hopf solutions satisfying the energy inequality almost everywhere for
positive moments of time, but not for the initial one, then a multivalued
semiflow is defined but we cannot obtain a bounded absorbing set. Therefore,
we are not able to prove the existence of a global attractor in this
situation. On the other hand, if we add the Leray-Hopf solutions satisfying
the energy inequality at the initial moment of time, then a bounded
absorbing set exists, but the translation of solutions fails to be a
solution, so we cannot define a multivalued semiflow. A detailed discussion
of these difficulties is given in \cite{BCKV}. 

The first result in this direction was published in \cite{FoiasTemam} (see 
\cite{FoiasRosaTemam} as well). The\ so-called universal attractor,
consisting of the bounded complete weak solutions, was defined. This set was
proved to attract uniformly in the weak topology of the phase space the
Leray-Hopf solutions starting at a bounded set and which satisfy the energy
inequality at the initial moment of time. In such a situation, as remarked
before, a semiflow cannot be defined and the attractor was shown to be
positively invariant but not negatively invariant. This problem was overcome
in \cite{KapVal07} (see also \cite{CheskidovLu}) by restricting the phase
space to a suitable ball, proving the existence of the weak global attractor
in\ both the autonomous and nonautonomous cases. However, the question about
the convergence of the solutions starting outside this ball to the global
attractor remained open. 

In this paper, in order to define a weak global attractor in the whole phase
space, we follow the approach of \cite{CarChoDlo} by taking suitable
approximations. In this way, we are able to define a global attractor which
uniformly attracts, with respect to the weak topology, the solutions
starting at any bounded set. Moreover, the solutions define a multivalued
semiflow in the attractor (although not in the whole space) and this set is
negatively and positively invariant with respect to it.

In order to be more specific about the results proved here we consider the
Navier-Stokes equations (NSE) 
\begin{equation}  \label{eq:NS}
\left\{ 
\begin{array}{ll}
u_{t}-\nu\Delta u+(u\cdot\nabla)u=-\nabla p+f, & t>0, \ x\in\Omega, \\ 
\mathrm{div} \ u=0, & t>0, \ x\in\Omega, \\ 
u=0, & t>0, \ x\in\partial\Omega, \\ 
u(0,x)=u_0(x), & x\in\Omega,%
\end{array}
\right.
\end{equation}
where $\Omega\subset\mathbb{R}^{3}$ is a bounded domain with smooth boundary 
$\partial \Omega$, $u=(u_1,u_2,u_3)$ is the unknown velocity field of the
fluid, $\nu>0$ is a given constant kinematic viscosity, $p$ is the unknown
pressure, $u_0$ is a initial velocity field and $f$ is a constant external
force taken in $L^2(\Omega)^3$.

We also consider the globally modified Navier-Stokes equations (GMNSE) 
\begin{equation}  \label{NavierModified}
\left\{%
\begin{array}{ll}
u_t-\nu\Delta u+F_N(u)(u\cdot \nabla)u=-\nabla p+f, & t>0, \ x\in \Omega, \\ 
\mathrm{div}\ u=0, & t>0, \ x\in \Omega, \\ 
u=0, & t>0, \ x\in \partial \Omega \\ 
u(0,x)=u_0(x), & x\in \Omega,%
\end{array}%
\right.
\end{equation}
where $N$ is a fixed positive number and $F_N(u) = f_N(\|u\|_{L^4(\Omega)})$, where $%
f_N\colon [0,\infty)\to [0,\infty)$ is defined by 
\begin{equation}  \label{defFN}
f_N(r) =\min\left\{1,\frac Nr\right\} \quad \hbox{ for } r\geqslant 0,
\end{equation}
 which were
introduced in \cite{CKR} with the term $f_N(\|u\|_V)$ instead of $f_N(\|u\|_{L^4(\Omega)})$. It follows in a similar way as in \cite{CKR, Romito}, for the case $f_N(\|u\|_V)$, that
problem \eqref{NavierModified} generates a family of continuous semigroups $%
\{S_N(t)\colon t\geqslant 0\}$ in $H$ given by $S_N(t)u_0 =u_N(t,u_0)$,
where $u_N(\cdot,u_0)$ is the unique classical solution to %
\eqref{NavierModified}. We will give a different approach to obtain the
semigroup $\{S_N(t)\colon t\geqslant 0\}$ and the existence of its global
attractors, which will provide additional regularity of solutions and of the
global attractors.

Following \cite{Lions69,Temam77,Temam95}, if 
\begin{equation*}
\mathcal{V} = \{u\in C_0^\infty(\Omega)^3\colon \mathrm{div} \ u = 0\}, 
\end{equation*}
where $C_0^\infty(\Omega)^3$ is the set of smooth functions $u\colon
\Omega\to \mathbb{R}^3$ with compact support in $\Omega$ we define 
\begin{equation*}
H = \hbox{ the closure of } \mathcal{V} \hbox{ in } L^2(\Omega)^3, 
\end{equation*}
with inner product $(\cdot,\cdot)$ and norm $\|\cdot\|_H$ where, for $u,v\in
L^2(\Omega)^3$, 
\begin{equation*}
(u,v) = \sum_{j=1}^3 \int_\Omega u_j v_j. 
\end{equation*}
and 
\begin{equation*}
V = \hbox{ the closure of } \mathcal{V} \hbox{ in } H^1_0(\Omega)^3, 
\end{equation*}
with scalar product $((\cdot,\cdot))$ and norm $\|\cdot\|_V$ where, for $%
u,v\in H^1_0(\Omega)^3$, 
\begin{equation*}
((u,v)) = \sum_{i,j=1}^3 \int_\Omega \frac{\partial u_j}{\partial x_i}\frac{%
\partial v_j}{\partial x_i}. 
\end{equation*}

Consider the Leray projection $P\colon L^2\to H$ and $A= P\Delta \colon D(A)
\subset H\to H$ the Stokes operator (we have $D(A)\subset W^{2,2}(\Omega)^3
\cap H$, see \cite{Olga63} for instance). Since $A$ is closed and densely
defined ($A$ is, in fact, sectorial), we can consider the
interpolation-extrapolation scale generated by $(H,A)$, given by 
\begin{equation*}
\{(H_\alpha,A_\alpha)\colon \alpha \geqslant -1\}. 
\end{equation*}
Since $H$ is reflexive, thanks to \cite[Theorem V.1.5.12]{amann}, we have 
\begin{equation}\label{eq.dual}
H_{-\alpha}=(H_\alpha)^{\ast} \quad \hbox{ for } \alpha \in [0,1]. 
\end{equation}
In this way, we have found $A_{-1}$, with $D(A_{-1})=H$, 
\begin{equation}  \label{defA-1}
A_{-1}\colon H\subset H_{-1} \to H_{-1},
\end{equation}
satisfying $A_{-1}u\, (\phi) = {\displaystyle\int_\Omega} u A\phi$ for each $%
u\in H$ and $\phi\in D(A)$, and this operator is sectorial (see \cite%
{VonWhal}, for instance).

With this setting, we prove the following result concerning the local well
posedness of \eqref{NavierModified}, which will be proved at the end of
Section \ref{vwform_gmnse}.

\begin{theorem}
\label{Main1} Assume that $Pf\in H_{-\frac12}$. For each $N>0$, equation %
\eqref{NavierModified} generates a semigroup $S_N=\{S_N(t)\colon t\geqslant
0\}$ in $H$ with a global attractor $\mathcal{A_N}$. For each $u_0\in H$, if 
$u(t) = S_N(t)u_0$ for $t\geqslant 0$, we have 
\begin{equation*}
u \in C([0,\infty),H)\cap C((0,\infty),V)\cap C^1((0,\infty),H_{-\frac12}),
\end{equation*}
and $u$ is a classical solution of \eqref{NavierModified}. Moreover, for any $T>0$, $0<\eta<\frac{1}{8}$ and $p(\frac12+\eta)<1$, 
\begin{equation*}
\frac{du}{dt} \in L^p(0,T;H_{\eta-\frac{1}{2}}). 
\end{equation*}
In particular, for any $1\leqslant p<\frac{8}{5}$, $(0,T]\ni t \mapsto \frac{%
du}{dt}(t) \in H_{- \frac38}$ is locally H\"older continuous and is in $L^p%
\big(0,T;H_{-\frac38}\big)$.
\end{theorem}

Now we introduce some terminology to be able to state our main result
concerning the existence of a weak global attractor for \eqref{eq:NS}. First
we define the set (candidate for a global attractor) $\mathcal{A}\subset H$
as follows 
\begin{equation*}
\mathcal{A} = \left\{ 
\begin{aligned}
 y\in H \colon \hbox{ there are sequences } & t_j\stackrel{j\to\infty}{\longrightarrow}\infty,\ \{u_0^j\}_{j\in \mathbb{N}} \subset B_0 \\ \hbox{ and } N_j\stackrel{j\to\infty}{\longrightarrow} \infty
&\text{ such that }S_{N_j}(t_j)u_0^j\to y\text{ weakly in } H.
\end{aligned}
\right\}, 
\end{equation*}
where $B_0$ is a suitable closed ball in $H$ (see \eqref{def.B0}).

Let $\mathcal{K}$ be the set of all weak solutions of \eqref{eq:NS} (we
recall its meaning in Definition \ref{def.weaksol}). We now define the
subset of weak solutions that will be of our interest.

\begin{definition}
\label{KN} We say that $u\in\mathcal{KN}$ if $u\in\mathcal{K}$ and one of
the following conditions holds:

\begin{enumerate}
\item $u(0)\in\mathcal{A}$ and there exist $t_j\overset{j\to\infty}{%
\longrightarrow}\infty$, $\{u_0^j\}_{j\in \mathbb{N}}\subset B_0$ and $N_j%
\overset{j\to\infty}{\longrightarrow} \infty$ such that for any $%
t_0\geqslant0$ and $[0,\infty)\ni s_j\overset{j\to\infty}{\longrightarrow}
t_0$ we have%
\begin{equation*}
S_{N_j}(s_j+t_j)u_0^j\overset{j\to\infty}{\longrightarrow} u(t_0)\text{
weakly in }{H}. 
\end{equation*}

\item $u(0)\in{H}\setminus\mathcal{A}$ and there exists $N_j\overset{%
j\to\infty}{\longrightarrow} \infty$ such that for any $t_0\geqslant0$ and $%
[0,\infty)\ni s_j\overset{j\to\infty}{\longrightarrow} t_0$ we have%
\begin{equation*}
S_{N_j}(s_j)u(0)\overset{j\to\infty}{\longrightarrow} u(t_0)\text{ weakly in 
}{H}. 
\end{equation*}
\end{enumerate}
\end{definition}

With this we prove the following result.

\begin{theorem}
\label{Attractor} Assume that $Pf\in H_{-\frac12}$. The set $\mathcal{A}$ is a weak global attractor for the
solutions in $\mathcal{KN}$, which means that:

\begin{enumerate}[label={\rm (\alph*)}]

\item \label{Attractor-a} $\mathcal{A}$ is weakly compact in $H$.

\item \label{Attractor-b} $\mathcal{A}$ is negatively invariant, that is,
for any $y\in\mathcal{A}$ and $t\geqslant0$ there exists $u \in\mathcal{KN} $
such that $u(0) \in\mathcal{A}$ and $u(t)=y$.

\item \label{Attractor-c} $\mathcal{A}$ is positively invariant, that is,
given $t\geqslant0$ we have 
\begin{equation*}
\{u(t)\colon u\in\mathcal{KN}, \ u(0)\in\mathcal{A}\}\subset\mathcal{A}. 
\end{equation*}

\item \label{Attractor-d} $\mathcal{A}$ weakly attracts $\mathcal{KN}$, that
is, for any bounded subset $B$ of $H$ we have 
\begin{equation*}
\sup_{u\in\mathcal{KN},\ u(0) \in B}dist_w (u(t),\mathcal{A}) \overset{%
t\to\infty}{\longrightarrow} 0, 
\end{equation*}
where $dist_w(\cdot,\cdot)$ is the Hausdorff semidistance defined using the
metric induced by the weak topology of $H$ in the ball $B_0$ (see %
\eqref{def.distw}).
\end{enumerate}
\end{theorem}

Finally, we obtain that the global attractors for the globally modified
Navier-Stokes equations behave upper semicontinuously with respect the weak
global attractor $\mathcal{A}$.

The paper is organized as follows: in Section \ref{lecmaps} we introduce the
notion of $\epsilon$-regular maps and a local existence result for
semilinear parabolic problems with nonlinearities being $\epsilon-$regular
maps. In addition, we present a new result (see Theorem \ref{theorem.Regular}%
) regarding the regularity of solutions. In Section \ref{vwform_gmnse} we
prove that the nonlinearities associated to the GMNSE are $\frac12$-regular
maps, and obtain the local existence and regularity results for the very
weak formulation of \eqref{NavierModified}, i.e. Theorem \ref{Main1} will be
proved. Finally, in Section \ref{weak-ga}, we obtain the result on existence
of the weak global attractor, prove Theorem \ref{Attractor} and obtain the
upper semicontinuity of the attractors (see Proposition \ref{prop.Cont}).

\section{Local existence and regularity results for abstract semilinear
parabolic problems}

\label{lecmaps}

In what follows we will introduce some notation and terminology in order to
state the main results of \cite[Section 2]{AC} that will be used to obtain
the global well posedness, existence of the global attractor and its
regularity for \eqref{NavierModified}, and also present new regularity
results (Lemma \ref{lemma.Regular} and Theorem \ref{theorem.Regular}).

Consider $-A\colon D(A)\subset X\to X$ a positive sectorial operator, $%
X^{\alpha}=(D(A^{\alpha}),\|\cdot\|_{X^{\alpha}})$, $\alpha\geqslant0$ the
scale of fractional power spaces associated with $A$, where $%
\|x\|_{X^{\alpha}} = \|A^{\alpha}x\|_{X}$ for $x\in X^{\alpha}$, and $%
\{e^{At}\colon t \geqslant0\}$ the analytic semigroup generated by $A$.
Without loss of generality, we can assume that there exists a constant $%
M\geqslant1$ such that 
\begin{equation*}
\|e^{At}x\|_{X^{\alpha}} \leqslant M\|x\|_{X^{\alpha}} \quad\hbox{ for }
t\geqslant0, \ x\in X^{\alpha}\hbox{ and } \alpha\geqslant0, 
\end{equation*}
and 
\begin{equation*}
\|e^{At}x\|_{X^{\alpha}} \leqslant Mt^{-\alpha}\|x\|_{X} \quad\hbox{ for }
t>0, \ x\in X \hbox{ and } \alpha\geqslant0. 
\end{equation*}
With this, we can see that, for $0\leqslant\alpha\leqslant\beta$ we have 
\begin{equation}  \label{est.e}
\|e^{At}x\|_{X^{\beta}} \leqslant Mt^{\alpha- \beta}\|x\|_{X^{\alpha}} \quad%
\hbox{ for all } t>0 \hbox{ and } x\in X^{\alpha}.
\end{equation}

Given $\epsilon$, $\rho$ and $\gamma(\epsilon)$ positive constants, with $%
\rho>1$ and $\rho\epsilon\leqslant\gamma(\epsilon)<1$, we define $\mathcal{F}%
(\epsilon,\rho,\gamma(\epsilon))$ as the family of functions $f\colon
X^{1+\epsilon}\to X^{\gamma(\epsilon)}$, such that, for all $x,y\in
X^{1+\epsilon}$, satisfy 
\begin{equation}  \label{epslips}
\|f(x)-f(y)\|_{X^{\gamma(\epsilon)}} \leqslant c \|x-y\|_{X^{1+\epsilon}}
(\|x\|^{\rho-1}_{X^{1+\epsilon}}+ \|y\|^{\rho -1}_{X^{1+\epsilon}}+1)
\end{equation}
and 
\begin{equation}  \label{epsbound}
\|f(x)\|_{X^{\gamma(\epsilon)}} \leqslant c (\|x\|^{\rho
}_{X^{1+\epsilon}}+1),
\end{equation}
for some constant $c>0$. A function $f\in \mathcal{F}(\epsilon,\rho,\gamma(%
\epsilon))$ is called an \textbf{$\epsilon$-regular map} relative to the
pair $(X^1,X^0)$.

We consider the local well posedness for the abstract parabolic problem 
\begin{equation}  \label{abseqo}
\left\{ \begin{aligned} &\dot x=Ax+f(x), \ t>0\\ &x(0)=x_0, \end{aligned} %
\right.
\end{equation}
when $f\in \mathcal{F}(\epsilon,\rho,\gamma(\epsilon))$. We say that a
function $x\colon[0,\tau]\to X^{1}$ is an \textbf{$\epsilon$-regular mild
solution} of \eqref{abseqo} in $[0,\tau]$ if $x\in C([0,\tau],X^1)\cap
C((0,\tau],X^{1+\epsilon})$ and 
\begin{equation}  \label{eq.VCF}
x(t) = e^{At}x_0 + \int_0^{t} e^{A(t-s)}f(x(s))ds \quad\hbox{ for each } t\in%
[0,\tau].
\end{equation}

With these definitions, we have the following:

\begin{theorem}[{See {\protect\cite[Section 2]{AC}}}]
\label{locexist} Let $f\in \mathcal{F}(\epsilon,\rho,\gamma(\epsilon))$.
Given $y_0\in X^1$, there exist $r>0$ and $\tau_0>0$ such that for any $%
x_0\in B_{X^1}(y_0,r)$ there exists a continuous function $x(\cdot,x_0)\colon%
[0,\tau_0]\to X^1$, which is the unique $\epsilon$-regular mild solution of %
\eqref{abseqo}. This solution satisfies 
\begin{equation*}
x\in C((0,\tau_0],X^{1+\theta}) \quad \hbox{ for } \quad 0\leqslant\theta<
\gamma(\epsilon), 
\end{equation*}
and 
\begin{equation*}
t^{\theta}\|x(t,x_0)\|_{X^{1+\theta}}\xrightarrow{t \to 0^+} 0 \quad \hbox{
for } \quad 0<\theta< \gamma(\epsilon). 
\end{equation*}
Moreover, given $0\leqslant\theta_0 < \gamma(\epsilon)$, there exists $C>0$
such that if $x_0, z_0\in B_{X^1}(y_0,r)$ then 
\begin{equation*}
t^{\theta}\|x(t,x_0) - x(t,z_0)\|_{X^{1+\theta}} \leqslant
C\|x_0-z_0\|_{X^1}, 
\end{equation*}
for all $t\in[0,\tau_0]$ and $0\leqslant\theta\leqslant\theta_0$. 

Also, 
if $\gamma(\epsilon)>\rho\epsilon$, the time of existence is uniform on
bounded subsets of $X^1$ and, furthermore, 
\begin{equation*}
x\in C((0,\tau_0], X^{1+\gamma(\epsilon)})\cap C^1((0,\tau_0],
X^{\gamma(\epsilon)}), 
\end{equation*}
$x$ verifies \eqref{abseqo} and, if $\tau_m$ is the maximal time of
existence for $x(t,x_0)$, either $\tau_m=\infty$ or $\lim\limits_{t\to%
\tau_m^-}\|x(t,x_0)\|_{X^1}=\infty$. 
\end{theorem}

\begin{remark}
We stress that \textit{it is not assumed that $f$ is defined on $X^1$}. The
only requirement on $f$ is that it is an $\epsilon$-regular map relative to
the pair $(X^1,X^0)$ for some $\epsilon>0$. Hence, local well posedness in $%
X^1$ is obtained without requiring that the nonlinearity $f$ is defined on $%
X^1$.
\end{remark}

\subsection{Regularity results}

Continuing the work done in \cite{AC}, we present a new result regarding the
regularity of $\epsilon$-regular mild solutions of \eqref{abseqo}. We begin
with a set of technical lemmas.

\begin{lemma}
\label{lemma.Aux.1} If $x\colon [0,\tau_0]\to X^1$ is an $\epsilon$-regular
mild solution of \eqref{abseqo} then for each $0\leqslant \theta <
\gamma(\epsilon)$ there exists a constant $C\geqslant 0$ such that 
\begin{equation*}
\|x(t)\|_{X^{1+\theta}}\leqslant Ct^{-\theta} \quad\hbox{ for all }
t\in(0,\tau_0]. 
\end{equation*}
\end{lemma}

\begin{proof}
Since $x$ is an $\epsilon$-regular mild solution of \eqref{abseqo}, from
Theorem \ref{locexist} we obtain 
\begin{equation*}
\sup_{s\in (0,\tau_0]} s^\epsilon \|x(s)\|_{X^{1+\epsilon}} < \infty. 
\end{equation*}
Hence, using \cite[Lemma 2]{AC}, we have 
\begin{equation*}
t^\theta \left\| \int_0^t e^{A(t-s)}f(x(s))ds \right\|_{X^{1+\theta}}
\leqslant C_1, 
\end{equation*}
for all $t\in (0,\tau_0]$ and some constant $C_1\geqslant 0$. Thus, using %
\eqref{eq.VCF} and \eqref{est.e}, we have 
\begin{equation*}
\|x(t)\|_{X^{1+\theta}} \leqslant\|e^{At}x_0\|_{X^{1+\theta}}+ \left\|
\int_0^t e^{A(t-s)}f(x(s))ds\right\|_{X^{1+\theta}} \leqslant
Mt^{-\theta}\|x_0\|_{X^1}+C_1 t^{-\theta}, 
\end{equation*}
and the estimate follows, taking $C=M\|x_0\|_{X^1} + C_1$.
\end{proof}

\begin{lemma}
\label{lemma.Aux.1.1} Fix $0<T<\infty$ and let $g\colon (0,T] \to
X^{\gamma(\epsilon)}$ be locally H\"older continuous with $\int_0^r
\|g(s)\|_{X^{\gamma(\epsilon)}} ds < \infty$ for some $r>0$. For $0\leqslant
t \leqslant T$, define 
\begin{equation}  \label{def.G}
G(t) = \int_0^t e^{A(t-s)}g(s)ds.
\end{equation}
Then $G$ is continuous on $[0,T]$, continuously differentiable on $(0,T)$,
with $G(t)\in D(A)$ for $t\in (0,T)$ and 
\begin{equation*}
\frac{dG}{dt}(t) = AG(t) + g(t) \quad \hbox { for } t\in (0,T), 
\end{equation*}
with $G(t)\xrightarrow{t\to 0^+} 0$ in $X^{\gamma(\epsilon)}$.
\end{lemma}

\begin{proof}
Apply \cite[Lemma 3.2.1]{Henry} with $g$ and $X^{\gamma(\epsilon)}$ in place
of $f$ and $X$.
\end{proof}

\begin{lemma}
\label{lemma.Regular} Fix $0<T<\infty$, and assume that $g\colon(0,T]\to
X^{\gamma(\epsilon)}$ satisfies 
\begin{equation*}
\| g(\theta)-g(r)\|_{X^{\gamma(\epsilon)}} \leqslant K(r)|\theta-r|^{\delta}
\quad \hbox{ for all } \theta,r\in (0,T], 
\end{equation*}
where $\delta>0$ is a fixed constant and $K\colon(0,T]\to[0,\infty)$ is
continuous, and integrable in $(0,T)$. Then the function $G$ defined in %
\eqref{def.G} is continuously differentiable from $(0,T)$ into $%
X^{\gamma(\epsilon)+\mu}$ for any $0\leqslant\mu < \delta$, and 
\begin{equation}  \label{eq.required1}
\left\|\frac{dG}{dt}(t)\right\|_{X^{\gamma(\epsilon)+\mu}} \leqslant
Mt^{-\mu}\|g(t)\|_{X^{\gamma(\epsilon)}}+M\int_0^{t} K(s)(t-s)^{-\mu
+\delta-1}ds \quad \hbox{ for } t\in (0,T),
\end{equation}
with $M$ independent from $\mu$ and $g$. Furthermore 
\begin{equation*}
(0,T) \ni t \mapsto\frac{dG}{dt}(t) \in X^{\gamma(\epsilon)+\mu} 
\end{equation*}
is locally H\"older continuous, provided $\int_0^{r} K(s)ds = O(r^{\varphi })
$ as $r\to 0^+$, for some $\varphi>0$.
\end{lemma}

\begin{proof}
From Lemma \ref{lemma.Aux.1.1} we know that 
\begin{equation*}
\frac{dG}{dt}(t) = AG(t) + g(t) = e^{At}g(t) + H(t) \quad \hbox{ for all }
t\in (0,T), 
\end{equation*}
where $H(t) = \displaystyle \int_0^{t} Ae^{A(t-s)}(g(s)-g(t))ds$. Hence, for 
$0\leqslant\mu< \delta$, we have 
\begin{align*}
\left\| \frac{dG}{dt}(t)\right\| _{X^{\gamma(\epsilon)+\mu}} & \leqslant
\|e^{At}g(t)\|_{X^{\gamma(\epsilon)+\mu}} + \int_0^{t}
\|Ae^{A(t-s)}(g(s)-g(t))\|_{X^{\gamma(\epsilon)+\mu}}ds \\
& \leqslant Mt^{-\mu}\|g(t)\|_{X^{\gamma(\epsilon)}} + M \int_0^{t}
K(s)(t-s)^{-\mu+\delta-1}ds,
\end{align*}
and \eqref{eq.required1} follows. For the second part fix $r_0>0$ such that $%
\int_0^r K(s)ds \leqslant Cr^{\varphi}$, for $0\leqslant r\leqslant r_0$.
Hence, for a fixed $t>0$ and $0<h\leqslant r_0$, we have 
\begin{align*}
H & (t+h)-H(t) = \int_0^{t+h}Ae^{A(t+h-s)}(g(s)-g(t+h))ds - \int_0^{t}
Ae^{A(t-s)}(g(s)-g(t))ds \\
& = \int_0^{h} Ae^{A(t+h-s)}(g(s)-g(t+h))ds + \int_0^{t} Ae^{A(t-s)}%
\underbrace{(g(s+h)-g(s)-g(t+h)+g(t))}_{(\ast)}ds.
\end{align*}
For the term $(\ast)$, we obtain the estimates 
\begin{equation}  \label{gshth}
\begin{aligned} \|g(s&+h)-g(s) -g(t+h)+ g(t)\|_{X^{\gamma(\epsilon)}} \\ &
\leqslant
\|g(s+h)-g(s)\|_{X^{\gamma(\epsilon)}}+\|g(t+h)-g(t)\|_{X^{\gamma(%
\epsilon)}} \leqslant K(s)h^\delta + K(t)h^\delta, \end{aligned}
\end{equation}
and 
\begin{equation}  \label{gshth2}
\begin{aligned} \|g(s+h)-g(s)-g(t+h&)+g(t)\|_{X^{\gamma(\epsilon)}} \\ &
\leqslant
\|g(s+h)-g(t+h)\|_{X^{\gamma(\epsilon)}}+\|g(t)-g(s)\|_{X^{\gamma(%
\epsilon)}} \\ & \leqslant K(s+h)(t-s)^\delta + K(s)(t-s)^\delta.
\end{aligned}
\end{equation}
Choosing $\omega \in(0,1)$ such that $\mu< \delta\omega$, interpolating %
\eqref{gshth} and \eqref{gshth2} we obtain 
\begin{align*}
\|g(s+h)-g&(s)-g(t+h)+g(t)\|_{X^{\gamma(\epsilon)}} \\
& \leqslant (K(s)+K(t))^{1-\omega}(K(s+h)+K(s))^{\omega}h^{(1-\omega)\delta}
(t-s)^{\delta\omega}.
\end{align*}

Thus 
\begin{align*}
\|H&(t+h)-H(t)\|_{X^{\gamma(\epsilon)+\mu}} \\
& \leqslant M\int_0^h K(s)(t+h-s)^{-\mu+\delta-1} ds \\
& \hspace{30pt} + Mh^{\delta(1-\omega)}\int_0^t
(K(s)+K(t))^{1-\omega}(K(s+h)+K(s))^{\omega} (t-s)^{-\mu+\delta\omega-1} ds
\\
& \leqslant MCt^{-\mu+\delta-1}h^{\varphi} \\
& \hspace{30pt} + Mh^{\delta(1-\omega)}\int_0^{t}
(K(s)+K(t))^{1-\omega}(K(s+h)+K(s))^{\omega} (t-s)^{-\mu+\delta\omega-1}ds.
\end{align*}
Since 
\begin{equation*}
\int_0^{t} (K(s)+K(t))^{1-\omega}(K(s+h)+K(s))^{\omega} (t-s)^{-\mu
+\delta\omega-1}ds <\infty 
\end{equation*}
for $t\in (0,T)$ fixed, and the map $(0,T)\ni t \mapsto e^{At}g(t) \in
X^{\gamma(\epsilon)+\mu}$ is locally H\"older continuous, the result follows.
\end{proof}

Lastly, we will need the following singular version of the Gr\"onwall
inequality, which proof is analogous to \cite[Lemma 7.1.1]{Henry}.

\begin{lemma}[Singular Gr\"onwall Inequality]
\label{Lgronwall} Let $a,b,c\geqslant 0$, $0\leqslant \alpha,\beta,\gamma <1$
and $u\colon(0,T)\to \mathbb{R}$ be an integrable function with 
\begin{equation}  \label{Egronwall}
0\leqslant u(t) \leqslant a t^{-\alpha}+ bt^{-\beta}+c\int_0^t
(t-s)^{-\gamma}u(s)ds \quad \hbox{ for a.e. } t\in (0,T).
\end{equation}
Then, there exists a constant $K\geqslant 0$ that depends only on $c,\gamma,T
$ such that 
\begin{equation*}
u(t)\leqslant \frac{K}{1-\alpha} a t^{-\alpha} + \frac{K}{1-\beta} b
t^{-\beta} \quad \hbox{ for a.e. } t\in (0,T). 
\end{equation*}
\end{lemma}

With these lemmas, we can present the main result of this section.

\begin{theorem}
\label{theorem.Regular} Let $f$ be an $\epsilon$-regular map relative to the
pair $(X^1,X^0)$ and $x \colon[0,\tau_0] \to X^1$ be the unique $\epsilon$%
-regular mild solution of \eqref{abseqo}. Then for each $0 <
\eta<\gamma(\epsilon)-\epsilon$ the map $(0,\tau_0]\ni t \mapsto\dfrac{dx}{dt%
}(t) \in X^{\gamma(\epsilon)+\eta}$ is locally H\"older continuous and there
exists a constant $C\geqslant 0$ such that 
\begin{equation*}
\Big\| \frac{dx}{dt}(t)\Big\|_{X^{\gamma(\epsilon)+\eta}} \leqslant C
t^{-\gamma(\epsilon)-\eta} \quad \hbox{ for } t\in (0,\tau_0]. 
\end{equation*}
In particular, $(0,\tau_0]\ni t \mapsto \frac{dx}{dt}(t)\in X^{\gamma(\epsilon)+\eta}$ is in $L^p(0,\tau_0;X^{\gamma(\epsilon)+\eta})$, provided that $p(\gamma(\epsilon)+\eta)<1$.
\end{theorem}

\begin{proof}
Fix $\tau\in(0,\tau_0)$ and, for $t\in(0,\tau_0-\tau]$, we define $g(t) =
f(x(t+\tau))$. Since $x\in C([\tau,\tau_0],X^{1+\epsilon})$ and $f$ is an $%
\epsilon$-regular map relative to the pair $(X^1,X^0)$, for $%
t,s\in(0,\tau_0-\tau]$ we have 
\begin{align*}
\|g(t)-g(s)\|_{X^{\gamma(\epsilon)}}&
=\|f(x(t+\tau))-f(x(s+\tau))\|_{X^{\gamma (\epsilon)}} \\
& \leqslant
c\|x(t+\tau)-x(s+\tau)\|_{X^{1+\epsilon}}(1+\|x(t+\tau)\|_{X^{1+\epsilon}}^{%
\rho-1}+\|x(s+\tau)\|_{X^{1+\epsilon}}^{\rho-1}) \\
& \leqslant C_1 \|x(t+\tau)-x(s+\tau)\|_{X^{1+\epsilon}},
\end{align*}
and 
\begin{equation*}
\|g(t)\|_{X^{\gamma(\epsilon)}} = \|f(x(t+\tau))\|_{X^{\gamma(\epsilon)}}
\leqslant c(\|x(t+\tau)\|_{X^{1+\epsilon}}^\rho+1) \leqslant C_2, 
\end{equation*}
for some constants $C_1,C_2\geqslant0$. Now, for $0<t<t+h\leqslant\tau_0-\tau
$ we have 
\begin{align*}
x(t+\tau+h)-x(t+\tau) = [e^{Ah}-I] & e^{At}x(\tau) + \int_0^{h}
e^{A(t+h-s)}g(s)ds \\
& + \int_0^{t} e^{A(t-s)}(g(s+h)-g(s))ds.
\end{align*}
Since 
\begin{equation*}
[e^{Ah}-I] e^{At}x(\tau) = \int_0^{h} \frac{d}{dr}e^{A(t+r)}x(\tau)dr=
\int_0^{h} Ae^{A(t+r)}x(\tau)dr, 
\end{equation*}
we have 
\begin{equation*}
\|[e^{Ah}-I] e^{At}x(\tau)\|_{X^{1+\epsilon}} \leqslant Mht^{\eta-1}
\|x(\tau)\|_{X^{1+\epsilon+\eta}}, 
\end{equation*}
and for the second term we obtain 
\begin{equation*}
\int_0^h \|e^{A(t+h-s)}g(s)\|_{X^{1+\epsilon}}ds \leqslant MC_2ht^{\gamma
(\epsilon)-\epsilon-1}. 
\end{equation*}
For the last term, we have 
\begin{equation*}
\int_0^t \|e^{A(t-s)}(g(s+h)-g(s))\|_{X^{1+\epsilon}}ds \leqslant M\int_0^t
(t-s)^{\gamma(\epsilon)-\epsilon-1}\|g(s+h)-g(s)\|_{X^{\gamma
(\epsilon)}}ds, 
\end{equation*}
therefore 
\begin{align*}
\|g(t+h)-g(t)\|_{X^{\gamma(\epsilon)}} & \leqslant
C_1\|x(t+\tau+h)-x(t+\tau)\|_{X^{1+\epsilon}} \\
& \leqslant MC_1h
t^{\eta-1}\|x(\tau)\|_{X^{1+\epsilon+\eta}}+MC_1C_2ht^{\gamma(\epsilon)-%
\epsilon-1} \\
& \hspace{20pt} + MC_1\int_0^t (t-s)^{\gamma
(\epsilon)-\epsilon-1}\|g(s+h)-g(s)\|_{X^{\gamma(\epsilon)}}ds.
\end{align*}
From Lemma \ref{Lgronwall} there exists a constant $K\geqslant 0$ such that 
\begin{equation*}
\|g(t+h)-g(t)\|_{X^{\gamma(\epsilon)}} \leqslant Kh\big[t^{\eta-1}\|x(\tau)%
\|_{X^{1+\epsilon+\eta}}+t^{\gamma(\epsilon)-\epsilon-1}\big]. 
\end{equation*}

Setting $K(s) = K\big[s^{\eta-1}\|x(\tau)\|_{X^{1+\epsilon+\eta}}+s^{\gamma(%
\epsilon)-\epsilon-1}\big]$ for $s\in(0,\tau_0-\tau]$, we see that 
\begin{equation*}
\|g(t)-g(s)\|_{X^{\gamma(\epsilon)}} \leqslant K(s)(t-s) 
\end{equation*}
and, moreover, 
\begin{equation*}
\int_0^{r} K(s)ds = K\left[ \frac{r^{\eta}}{\eta}\|x(\tau)\|_{X^{1+\epsilon
+\eta}}+\frac{r^{\gamma(\epsilon)-\epsilon}}{\gamma(\epsilon)-\epsilon }%
\right]. 
\end{equation*}
Hence, it follows from Lemma \ref{lemma.Regular} with $\delta=1$ that 
\begin{equation}  \label{est.G.grande}
\begin{aligned} \left\|\frac{dG}{dt}(t)\right\|_{X^{\gamma(\epsilon)+\eta}}
& \leqslant Mt^{-\eta}\|g(t)\|_{X^{\gamma(\epsilon)}}+M\int_0^t
K(s)(t-s)^{-\eta}ds\\ & \leqslant MC_2t^{-\eta} +
MK\mathcal{B}(\eta,1-\eta)\|x(\tau)\|_{X^{1+\epsilon+\eta}} \\ &
\hspace{80pt} + MK
\mathcal{B}(\gamma(\epsilon)-\epsilon,1-\eta)t^{\gamma(\epsilon)-\epsilon-%
\eta}, \end{aligned}
\end{equation}
since $0<\eta<\gamma(\epsilon)-\epsilon <1$ (here $\mathcal{B}$ denotes the
beta function), and that $(0,T)\ni t\mapsto \frac{dG}{dt}(t) \in
X^{\gamma(\epsilon)+\eta}$ is locally H\"older continuous. Since 
\begin{equation*}
\|Ae^{At}x(\tau)\|_{X^{\gamma(\epsilon)+\eta}} \leqslant Mt^{-\gamma
(\epsilon)+\epsilon}\|x(\tau)\|_{X^{1+\epsilon+\eta}}, 
\end{equation*}
together with \eqref{est.G.grande} we obtain a constant $C_3\geqslant 0$
such that 
\begin{equation*}
\left\| \frac{dx}{dt}(t)\right\|_{X^{\gamma(\epsilon)+\eta}} \leqslant C_3%
\Big[t^{-\eta} +
t^{-\gamma(\epsilon)+\epsilon}\|x(\tau)\|_{X^{1+\epsilon+\eta}}+t^{\gamma(%
\epsilon)-\epsilon-\eta}\Big], 
\end{equation*}
and that $(0,T)\ni t \mapsto\frac{dx}{dt}(t)\in X^{\gamma(\epsilon)+\eta}$
is locally H\"older continuous. From Lemma \ref{lemma.Aux.1}, there exists a
constant $C\geqslant 0$ such that 
\begin{equation*}
\|x(\tau)\|_{X^{1+\epsilon+\eta}} \leqslant C\tau^{-\epsilon-\eta} 
\end{equation*}
and, therefore, there exists a constant $C_4\geqslant 0$ such that 
\begin{equation*}
\left\|\frac{dx}{dt}(t)\right\|_{X^{\gamma(\epsilon)+\eta}} \leqslant
C_4[t^{-\eta} +
t^{-\gamma(\epsilon)+\epsilon}\tau^{-\epsilon-\eta}+t^{\gamma(\epsilon)-%
\epsilon-\eta}]. 
\end{equation*}
Finally, taking $\tau=\frac{t}{2}$, we obtain 
\begin{equation*}
\left\| \frac{dx}{dt}(t)\right\| _{X^{\gamma(\epsilon)+\eta}} \leqslant
4C_4[t^{-\eta} +
t^{-\gamma(\epsilon)-\eta}+t^{\gamma(\epsilon)-\epsilon-\eta}] \leqslant
C_5t^{-\gamma(\epsilon)-\eta}, 
\end{equation*}
for all $t\in(0,\tau_0]$, for some constant $C_5\geqslant 0$.
\end{proof}

\section{Very weak formulation of the globally modified Navier-Stokes
equation}

\label{vwform_gmnse}

In this section we show that the nonlinearities of the GMNSE are $\frac38$%
-regular maps, obtain the very weak formulation of the GMNSE, prove Theorem %
\ref{Main1}, and guarantee the existence of the global attractor in $H$ for
the GMNSE.

\subsection{Nonlinearity of the GMNSE}

We begin with some results involving the map $f_N$ defined in \eqref{defFN},
to study the nonlinear term $F_N(u)(u\cdot \nabla u)$ of the GMNSE. We note
that, from \cite[Lemma 4]{CKR}, for every $N>0$ and $s,t\geqslant 0$ with $%
s+t>0$ we have 
\begin{equation}  \label{CKR-l4}
|f_{N}(s)-f_{N}(t)|\leqslant\frac{1}{\max\{s,t\}} |s-t|.
\end{equation}

Writing, for $r\geqslant 1$, $L^{r} =
L^{r}(\Omega)^{3}$, to simplify the notation, we point out that for each $N>0
$ and $u\in L^4$ we have 
\begin{equation}  \label{fNuN}
F_{N}(u)\|u\|_{L^4}\leqslant N,
\end{equation}
and for any $N>0$ and $u,v\in L^4(\Omega)$ with $\|u\|_{L^4(\Omega)}+\|v\|_{L^4(\Omega)}>0$ we have 
\begin{equation*}
|F_{N}(u)-F_{N}(v)|\leqslant\frac{1}{\max\{\|u\|_{L^4(\Omega)},\|v\|_{L^4(\Omega)}\}}
\|u-v\|_{L^4(\Omega)}, 
\end{equation*}
just using \eqref{CKR-l4} and noting that $|\|u\|_{L^4(\Omega)}-\|v\|_{L^4(\Omega)}|%
\leqslant\|u-v\|_{L^4(\Omega)}$. 

\medskip Before continuing, we note that the tensor product in $\mathbb{R}^3$ will be of great importance and we recall it here. The \textbf{tensor product} between $u=(u_1,u_2,u_3), \ v=(v_1,v_2,v_3) \in \mathbb{R}^3$ is the $3\times 3$ real matrix given by 
\begin{equation*}
u\otimes v=
\begin{pmatrix} 
u_1v_1 & u_1v_2 & u_1v_3\\
u_2v_1 & u_2v_2 & u_2v_3\\
u_3v_1 & u_3v_2 & u_3v_3
\end{pmatrix}=(u_iv_j)_{i,j=1}^3.
\end{equation*}
Clearly, with the following properties hold:

\begin{itemize}
\item $(u+v)\otimes w = u\otimes w + v\otimes w;$

\item $|u\otimes v|_{M_3({\mathbb{R}})}\leqslant |u|_{\mathbb{R}^3}\, |v|_{\mathbb{R}^3};$

\item $u\otimes v=(v\otimes u)^{t}$, where $A^{t}$ denotes the transposed
matrix of $A\in M_3(\mathbb{R})$.
\end{itemize}

\begin{lemma} \label{lemmaEstFN} 
For all $u,v\in L^4 $, we have 
\begin{equation*}
\|F_N(u) u\otimes u - F_N(v) v\otimes v\|_{L^2} \leqslant
3N\|u-v\|_{L^4}, 
\end{equation*}
and
\begin{equation*}
\|F_N(u) u\otimes u\|_{L^2} \leqslant N\|u\|_{L^4}.
\end{equation*}
\end{lemma}
\begin{proof}
Note that 
\begin{equation*}
\|F_{N}(u)u\otimes u\|_{L^2} \leqslant F_{N}(u)\|u\|_{L^4}\|u\|_{L^4} \leqslant N\|u\|_{L^4},
\end{equation*}
where in the last inequality we used \eqref{fNuN}.

Now we deal with the first estimate. Firstly, we treat the case where $%
\|u\|_{L^4},\|v\|_{L^4} \leqslant N$. In this case, we have $%
F_{N}(u)=F_{N}(v)=1$ and 
\begin{align*}
\|F_{N}(u) & u\otimes u - F_{N}(v) v\otimes v\|_{L^2}=\|u\otimes u -
v\otimes v\|_{L^2}  \leqslant\|u\otimes(u-v)\|_{L^2} +\|(u-v)\otimes v\|_{L^2} \\
& \leqslant\|u\|_{L^4}\|u-v\|_{L^4} + \|u-v\|_{L^4}\|v\|_{L^4} \leqslant  (\|u\|_{L^4}+\|v\|_{L^4})\|u-v\|_{L^4} \leqslant2N \|u-v\|_{L^4}.
\end{align*}

Now we assume that $\|v\|_{L^4}, \|u\|_{L^4}\geqslant N$. We can
assume, without loss of generality, that $\|v\|_{L^4}\geqslant
\|u\|_{L^4}$. If this case, we have 
\begin{align*}
\| & F_{N}(u) u\otimes u - F_{N}(v) v\otimes v\|_{L^2}=N\Big\|\frac {%
u\otimes u}{\|u\|_{L^4}} - \frac{v\otimes v}{\|v\|_{L^4}}\Big\|%
_{L^2} \\
& \leqslant N\frac{\|u\otimes(u-v)\|_{L^2}}{\|u\|_{L^4}} + N\|u\otimes
v\|_{L^2}\Big|\frac{1}{\|u\|_{L^4}}-\frac{1}{\|v\|_{L^{4}}}\Big| + N 
\frac{\|u\otimes(u-v)\|_{L^{2}}}{\|v\|_{L^{4}}} \\
& \leqslant N\frac{\|u\|_{L^{4}}\|u-v\|_{L^{4}}}{\|u\|_{L^{4}}} + N%
\|u\|_{L^{4}}\|v\|_{L^{4}}\frac{|\|u\|_{L^{4}}-\|v\|_{L^{4}}|}{%
\|u\|_{L^{4}}\|v\|_{L^{4}}}+N\frac{\|u\|_{L^{4}}\|u-v\|_{L^{4}}}{%
\|v\|_{L^{4}}} \\
& \leqslant 3 N\|u-v\|_{L^{4}}.
\end{align*}

For the last case, we can assume, without loss of generality, that 
$\|u\|_{L^{4}}\leqslant N \leqslant\|v\|_{L^{4}}$. Hence 
\begin{align*}
F_{N} & (u) u\otimes u - F_{N}(v) v\otimes v = \frac{1}{\|v\|_{L^{4}}}\big(
\|v\|_{L^{4}}u\otimes u - N v\otimes v\big) \\
& = \frac{1}{\|v\|_{L^{4}}}\big(\|v\|_{L^{4}} u\otimes
(u-v)+(\|v\|_{L^{4}}-N)u\otimes v + N(u-v)\otimes v\big).
\end{align*}
Estimating the $L^{2}$ norm and using the fact that $\|u\|_{L^{4}}\leqslant N$ and that 
\begin{equation*}
\|v\|_{L^{4}}-N \leqslant\|v\|_{L^{4}}-\|u\|_{L^{4}}
\leqslant\|u-v\|_{L^{4}}, 
\end{equation*}
we obtain 
\begin{equation*}
\|F_{N}(u)u\otimes u - F_{N}(v)v\otimes v\|_{L^{2}} \leqslant3N
\|u-v\|_{L^{4}}, 
\end{equation*}
and the proof is complete.
\end{proof}

We recall the following Sobolev type embeddings for a bounded domain $%
\Omega\subset\mathbb{R}^{3}$ with smooth boundary $\partial \Omega$. Given $%
s>0$ we consider $m$ the 
smallest integer greater than or equal to $s$. Then 
\begin{equation*}  \label{embedding-frationary-order-N}
[\, L^{p}(\Omega)^{3},W^{m,p}(\Omega)^{3}\,]_{s/m} = H^{s,p}(\Omega)^{3}.
\end{equation*}
Additionally, for $1\leqslant p\leqslant q= 3p/(3-sp)$, we have 
\begin{equation}  \label{embedding-frationary-order-N.2}
H^{s,p}(\Omega)^{3}\subset L^{q}(\Omega)^{3}.
\end{equation}

Hence, we arrive at the following result without having to know explicitly $%
H_\alpha$.

\begin{proposition}
\label{embbeding-scales-NSE} If $\alpha\in(0,1)$ and $q=\frac{6}{3-4\alpha}$%
, we have $H_\alpha \subset H^{2\alpha,2}(\Omega)^3$ and $H_\alpha \subset
L^{q}(\Omega)^{3}$.
\end{proposition}

\begin{proof}
By construction $D(A) \subset W^{2,2}(\Omega)^{3}$ and since $A$ is
sectorial, the inclusion holds with equivalent norms. Interpolating these
spaces we arrive at $H_\alpha \subset H^{2\alpha,2}(\Omega)^{3}$ and apply %
\eqref{embedding-frationary-order-N.2}.
\end{proof}


Now we define the spaces $X^1$ and $X^0$ that we use for the rest of the
paper. Set 
\begin{equation}  \label{defX}
X^1 = H \quad \hbox{ and } \quad X^0 = H_{-1}.
\end{equation}
Hence, choosing 
\begin{equation}  \label{defEpsilon}
\epsilon = \frac38 \quad \hbox{ and } \quad \gamma(\epsilon) = \frac12,
\end{equation}
we have 
\begin{equation*}
X^{\gamma(\epsilon)} = X^{\frac12} = H_{-\frac12} \quad \hbox{ and } \quad
X^{1+\epsilon} = X^{\frac{11}{8}} = H_{\frac38}.
\end{equation*}
Note that from Proposition \ref{embbeding-scales-NSE} we obtain 
\[
X^{1+\epsilon} = H_{\frac38} \subset L^4.
\]
Also, we point out that, from \eqref{eq.dual}, $X^{\gamma(\epsilon)}=H_{-\frac12} = (H_{\frac12})^\ast = V^\ast$.



\begin{theorem}
\label{non-linearity-ns} For each $N>0$, the map $G_N\colon X^{1+\epsilon}
\to X^{\gamma(\epsilon)}$ given by 
\begin{equation}  \label{defGN}
G_N(u)\phi = \int_\Omega F_N(u) u\otimes u\cdot \nabla \phi \quad \hbox{ for
} u\in X^{1+\epsilon} \hbox{ and } \phi \in V,
\end{equation}
is a well-defined $\frac38$-regular map relative to the pair $(X^1,X^0)$,
and we can choose any $\rho\in (1,\frac43]$. %
%
\end{theorem}

\begin{proof}
Since $X^{1+\epsilon} \subset L^4$, for $u\in X^{1+\epsilon}$ and $\phi\in
V$, using Lemma \ref{lemmaEstFN} 
we have 
\begin{align*}
|G_N(u)\phi| & = \Big|\int_{\Omega} F_{N}(u)u\otimes u \cdot\nabla\phi\Big| %
\leqslant\|F_{N}(u)u\otimes u\|_{L^{2}}\|\phi\|_V \\
& \leqslant N \|u\|_{L^4}\|\phi\|_V\leqslant Nk \|u\|_{H_\frac38}\|\phi\|_{V},
\end{align*}
where $k$ is the embedding constant of $%
H_{\frac38}$ into $L^4$. Then, we obtain 
\begin{equation*}
|G_N(u)\phi| \leqslant Nk\|u\|_{X^{1+\epsilon}}\|\phi\|_{V}, 
\end{equation*}
and thus 
\begin{equation*}
\|G_N(u)\|_{X^{\gamma(\epsilon)}} \leqslant Nk\|u\|_{X^{1+\epsilon}}. 
\end{equation*}

Again, from Lemma \ref{lemmaEstFN}, we obtain 
\begin{align*}
|G_N(u) \phi- G_N(v)\phi| & = \Big|\int_{\Omega}(F_{N}(u)u\otimes u
-F_{N}(v)v\otimes v)\cdot \nabla \phi\Big| \\
& \leqslant\|F_{N}(u)u\otimes u - F_{N}(v)v\otimes v\|_{L^{2}}\|\phi
\|_{H^1_0} \\
& \leqslant3N \|u-v\|_{L^4}\|\phi\|_{V} \leqslant
3Nk\|u-v\|_{X^{1+\epsilon}}\|\phi\|_{V}
\end{align*}
and, therefore 
\begin{equation*}
\|G_N(u)-G_N(v)\|_{X^{\gamma(\epsilon)}} \leqslant
3Nk\|u-v\|_{X^{1+\epsilon}}, 
\end{equation*}
and the result is complete.
\end{proof}

\subsection{The very weak formulation of the GMNSE}

Let us reformulate the GMNSE in a \textit{very weak context}. In the
framework of scales of Banach spaces, that consists basically in a
translation to the extrapolated scale.

Projecting the first equation of \eqref{NavierModified} into $H$,
formally multiplying it by a smooth function $\phi\colon \Omega\to \mathbb{R}%
^3$ with compact support in $\Omega$ and $\mathrm{div} \phi=0$, and
integrating over $\Omega$, we obtain 
\begin{align*}
& \frac{d}{dt}\int_{\Omega}u\cdot\phi=\int_{\Omega}A u\cdot\phi
-\int_{\Omega}F_N(u) (u\cdot \nabla)u \cdot \phi+
\int_{\Omega}f_{\sigma}\cdot\phi,
\end{align*}
where $f_\sigma = Pf$. 
We have 
\begin{equation*}
\int_\Omega A u\cdot\phi= \int_\Omega\Delta u\cdot\phi= \int_\Omega
u\cdot\Delta\phi = \int_\Omega u\cdot A\phi. 
\end{equation*}

Now we relate the tensor product $u\otimes u$ with $(u\cdot \nabla) u$. To
that end, we define 
\begin{equation*}
\mathrm{div} (u\, u) = \tfrac{\partial}{\partial x_1} (u_1 u)+\tfrac{\partial}{\partial x_2} (u_2 u) + \tfrac{%
\partial}{\partial x_3} (u_3 u) =(\mathrm{div} (u_1 u), \mathrm{div} (u_2u), \mathrm{div}
(u_3 u)), 
\end{equation*}
and note that, if $\mathrm{div}\, u=0$, we have 
\begin{equation}  \label{divuu}
(u\cdot\nabla) u = \tfrac{\partial}{\partial x_{1}} (u_{1} u)+ \tfrac{\partial}{\partial x_{2}} (u_{2} u)+ \tfrac{%
\partial}{\partial x_{3}} (u_{3} u) - \underbrace{\left( \tfrac {\partial}{%
\partial x_{1}} u_{1}+ \tfrac{\partial}{\partial x_{2}} u_{2}+ \tfrac{\partial}{\partial x_{3}} u_{3}\right) u%
}_{=0} = \mathrm{div} (u\, u).
\end{equation}

\begin{lemma}
\label{strong-to-weak} Let $u,\phi\colon\Omega\to \mathbb{R}^3$ be smooth
functions with compact support in $\Omega$ and $\mathrm{div} \ \phi=\mathrm{%
div} \ u=0$. Then 
\begin{equation*}
\int_{\Omega}(u\cdot\nabla) u\cdot\phi= - \int_{\Omega}u\otimes u \cdot
\nabla\phi, 
\end{equation*}
and the same holds for $u\in V$. In the last term, we mean the scalar product for matrices.
\end{lemma}

\begin{proof}
By \eqref{divuu} and the Divergence Theorem we have 
\begin{align*}
\int_{\Omega} & (u\cdot\nabla) u\cdot\phi = \int_{\Omega}\mathrm{div} ( u\,
u) \cdot\phi = \int_{\Omega}(\mathrm{div} (u_1 u), \mathrm{div} (u_2 u), \mathrm{div} (u_3 u))
\cdot(\phi_{1},\phi_2, \phi_3) \\
& =- \int_{\Omega} \big(u_1 u \cdot\nabla\phi_1+ u_2 u \cdot\nabla\phi_2+ (u_3 u)\cdot
\nabla\phi_3\big) \\
& \qquad \qquad + \int_{\partial \Omega} \underbrace{\big( \phi_1(u_
1u)_{\partial \Omega}+ \phi_2(u_2 u)_{\partial \Omega}+ \phi_3(u_3 u)_{\partial \Omega}\big)}_{=0}
d\sigma \\
& =-\int_{\Omega}[u_1 u\ \ u_2u \ \ u_3 u]\cdot[\nabla\phi_1\ \ \nabla \phi_2 \ \ 
\nabla\phi_3] \\
& =-\int_{\Omega} 
\begin{bmatrix}
u_1 u_1 & u_1u_2 & u_1 u_3 \\ 
u_2 u_1 & u_2u_2 & u_2u_3 \\ 
u_3 u_1 & u_3u_2 & u_3 u_3%
\end{bmatrix}
\cdot 
\begin{bmatrix}
\frac{\partial\phi_1}{\partial x_1} & \frac{\partial\phi_2}{\partial x_1} & \frac{\partial\phi_3}{%
\partial x_1} \\ 
\frac{\partial\phi_1}{\partial x_2} & \frac{\partial\phi_2}{\partial x_2} & \frac{\partial\phi_3}{\partial x_2} \\ 
\frac{\partial\phi_1}{\partial x_3} & \frac{\partial\phi_2}{\partial x_3} & \frac{\partial\phi_3}{%
\partial x_3}%
\end{bmatrix}
=- \int_{\Omega}u\otimes u \cdot\nabla\phi,
\end{align*}
where, for each $i=1,2, 3$, the term $(u_i u)_{\partial \Omega}$ is the
outward normal component of $u_i u$ in $\partial \Omega$. The last claim
follows from a simple density argument.
\end{proof}

From this lemma, formally we have 
\begin{equation*}
-\int_\Omega F_N(u)(u\cdot \nabla)u \cdot \phi= \int_{\Omega} F_N(u)
u\otimes u \cdot\nabla\phi, 
\end{equation*}
and, in this way, we obtain the \textit{very weak formulation} of the
globally modified Navier-Stokes equation 
\begin{equation}  \label{WNSE}
\frac{d}{dt}\int_\Omega u\cdot \phi =\int_\Omega u\cdot A\phi + \int_\Omega
F_N(u) u\otimes u\cdot \nabla\phi +\int_\Omega f_\sigma \cdot \phi, \quad
t>0.
\end{equation}

Therefore, \eqref{NavierModified} can be written in $H_{-1}$ as 
\begin{equation}  \label{VeryWeakNSE}
\left\{ \begin{aligned} & u_t = A_{-1}u + G_N(u) + f_{\sigma,-1}, \quad t>0
\\ & u(0)=u_0, \end{aligned} \right.
\end{equation}
where $A_{-1}$ is defined in \eqref{defA-1}, $G_N$ is defined in %
\eqref{defGN} and $f_{\sigma,-1} \in H_{-1}$ is the functional 
\begin{equation*}
H_1 \ni\phi\mapsto f_{\sigma,-1} \phi= \int_{\Omega}f_{\sigma} \cdot\phi. 
\end{equation*}


Once we arrive at this point, the local well posedness of %
\eqref{NavierModified} in $H$ follows from the fact that $G_N$ is a $\frac38$-regular map relative to the pair $(X^1,X^0)=(H,H_{-1})$, given that $f_{\sigma,-1}$ is a fixed vector in $H_{-1}$.

Observe that, for each $u_0\in H$ the $\frac38$-regular mild solution $%
u(\cdot,u_0)$ obtained in Theorem \ref{locexist} satisfies the variation of
constants formula associated to \eqref{VeryWeakNSE}, that is 
\begin{equation*}
u(t,u_0)=e^{A
t}u_0+\int_0^te^{A(t-s)}G_N(u(s,u_0))ds+\int_0^te^{A(t-s)}f_{\sigma,-1}ds. 
\end{equation*}
We make the following assumption 
\begin{equation}  \label{assumpf1}
f_{\sigma,-1}\in H_{-\frac12},
\end{equation}
and note that 
\begin{equation}  \label{GlobalExistence}
\begin{aligned} \|u(t,u_0)&\|_{H_\frac38}\leqslant \|e^{A
t}\|_{\mathcal{L}(H,H_{\frac38})}\|u_0\|_H +\int_0^t \|e^{A
(t-s)}\|_{\mathcal{L}(H_{-\frac12},H_{\frac38})}
\|G_N(u(s,u_0)\|_{H_{-\frac12}} ds\\ & \qquad \qquad + \int_0^t \|e^{A
(t-s)}\|_{\mathcal{L}(H_{-\frac12},H_{\frac38})}\|f_{\sigma,-1}\|_{H_{-%
\frac12}} ds,\\ &\leqslant Mt^{-\frac38}\|u_0\|_H
+8Mt^{\frac18}\|f_{\sigma,-1}\|_{H_{-\frac12}}+MN
k \int_0^t (t-s)^{-\frac78}\|u(s,u_0)\|_{H_{\frac38}} ds. \end{aligned}
\end{equation}
Hence, it follows by the singular Gr\"onwall inequality (Lemma \ref%
{Lgronwall}) that $u(\cdot,u_0)$ does not blow up in finite time in $%
H_{\frac38}$ and, in particular, it must exist for all time $t\geqslant 0$.

Therefore, for each $N>0$, provided \eqref{assumpf1} holds, problem %
\eqref{VeryWeakNSE} defines a semigroup $S_{N}=\{S_{N}(t)\colon t\geqslant0\}
$ in $H$, given by $S_{N}(t)u_0=u(t,u_0)$ where, for each $u_0\in H$, $%
[0,\infty)\ni t \mapsto u(t,u_0)\in H$ is the unique $\frac38-$regular
solution of \eqref{VeryWeakNSE}. Moreover, from the above estimate, $S_{N}(t)
$ a compact map for each $t>0$.

Hence, from Theorems \ref{locexist} and \ref{theorem.Regular}, the above
computations and Lemma \ref{AbsorbingN} below, we obtain Theorem \ref{Main1}.

\subsection{Global attractor in $H$ for the GMNSE}

To prove existence of a global attractor $\mathcal{A}_{N}$ for the semigroup 
$S_{N}$ associated with \eqref{VeryWeakNSE}, it suffices to show that there
exists a bounded absorbing set in $H$, since $S_N(t)$ is compact for each $%
t>0$. This has already been proved in \cite{CKR} when $f\in L^2(\Omega)^3$. When $f_{\sigma,-1}\in H_{-\frac12}$ it follows that
\begin{equation*}
\frac{1}{2}\frac{d}{dt}\| u\|_H^2+\nu\| u\|_{H_{\frac12}}^2
=\langle u,f_{\sigma,-1}\rangle_{\frac12,-\frac12}, 
\end{equation*}
which implies%
\begin{equation*}
\frac{1}{2}\frac{d}{dt}\| u\|_H^2 +\nu\| u\|_{V}^2 \leqslant \|
u\|_V \|f_{\sigma,-1}\|_{H_{-\frac12}}\leqslant \frac{\nu}{2}\| u\|_{V}^2 +%
\frac{1}{2\nu}\| f_{\sigma,-1}\|_{H_{-\frac12}}^2.
\end{equation*}
Consequently 
\begin{equation*}
\frac{1}{2}\frac{d}{dt}\| u\|_H^2 + \frac{\lambda_{1}\nu}{2}\|u\|_H^2
\leqslant \frac{1}{2}\frac{d}{dt}\|u\|_H^2 + \frac{\nu}{2}\|
u\|_{H_{\frac12}}^2 \leqslant \frac{1}{2\nu}\| f_{\sigma,-1}\|_{H_{-\frac12}}^2,
\end{equation*}
where $\lambda_{1}>0$ is the first eigenvalue of the Stokes operator $A$. This expression immediately gives us
\begin{equation}  \label{IneqAbsorb}
\|S_{N}(t)u_0\|_H^2\leqslant\|u_0\|_H^2e^{-\nu\lambda_{1} t}+\frac{%
\|f_{\sigma,-1}\|_{H_{-\frac12}}^2}{\lambda_1\nu^2 },
\end{equation}
for every $N>0$. Hence, the ball 
\begin{equation}  \label{def.B0}
B_0=\left\{u\in H \colon\|u\|_H^2\leqslant 1+\frac{\|f_{\sigma,-1}\|_{H_{-\frac12}}^2}{%
\lambda_1\nu^2}\right\}
\end{equation}
uniformly absorbs bounded subsets of $H$ under the action of the semigroup $%
S_{N}$, and we point out that $B_0$ is independent of $N$. More precisely,
we obtain immediately the following result.

\begin{lemma}
\label{AbsorbingN} Assume that $f_{\sigma,-1}\in H_{-\frac12}$. Then, for any bounded
set $B\subset H$ there exists a time $T_{B}>0$ such that 
\begin{equation*}
S_{N}(t)u_0 \in B_0 \quad\hbox{ for all } t\geqslant T_{B}, \ u_0\in B %
\hbox{ and } N>0. 
\end{equation*}
Therefore, for each $N>0$, the semigroup $S_N$ has a global attractor $%
\mathcal{A}_N\subset B_0$.
\end{lemma}

The regularity of the solutions (uniformly in bounded sets, see %
\eqref{GlobalExistence}) also promptly gives a bounded absorbing set in $%
H_{\frac38}$, but the latter depends on $N$. The global attractor $\mathcal{A%
}_N$ attracts bounded subsets of $H$ in the $H_{\frac38}$-norm, or even
stronger norms.

\section{Weak global attractors in \texorpdfstring{$H$}{H} for the
Navier-Stokes equation}

\label{weak-ga}
We will assume throughout this section that $f_{\sigma,-1}\in H_{-\frac12}$.
Let us recall the definition of the set $\mathcal{A}\subset H$ given in the
introduction. 
\begin{equation*}
\mathcal{A} = \left\{ 
\begin{aligned}
 y\in H \colon \hbox{ there are sequences } & t_j\stackrel{j\to\infty}{\longrightarrow}\infty,\ \{u_0^j\}_{j\in \mathbb{N}} \subset B_0\\
  \hbox{ and } N_j\stackrel{j\to\infty}{\longrightarrow} \infty
&\text{ such that }S_{N_j}(t_j)u_0^j\to y\text{ weakly in } H.
\end{aligned}
\right\}. 
\end{equation*}

Clearly $\mathcal{A}$ is nonempty, since for any given sequences $t_j\overset%
{j\to\infty}{\longrightarrow}\infty$, $\{u_0^j\}_{j\in \mathbb{N}}\subset B_0
$ and $N_j\overset{j\to\infty}{\longrightarrow} \infty$ there exists $%
j_{1}\in\mathbb{N}$ such that $S_{N_j}(t_j)u_0^j\in B_0$ for $j\geqslant
j_{1}$ (see Lemma \ref{AbsorbingN}), and hence there exists $y\in H$ such
that, up to a subsequence, $S_{N_j}(t_j)u_0^j\overset{j\to\infty}{%
\longrightarrow} y$ weakly in $H$. Thus, $y\in \mathcal{A}\neq\varnothing$.
Moreover, it is clear that $\mathcal{A}\subset B_0$.

\begin{definition}[Weak solution of \eqref{eq:NS}]
\label{def.weaksol} We say that the function 
\begin{equation*}
u\in L^\infty(0,T;H)\cap L^2(0,T;V)\quad \hbox{ with } \quad \frac{du}{dt}
\in L^1(0,T;V^*) 
\end{equation*}
is a weak solution of \eqref{eq:NS} on $(0,T)$, if 
\begin{equation}  \label{def_wsnse}
\frac{d}{dt}\int_\Omega u\cdot \phi =\int_\Omega A^\frac{1}{2}u\cdot A^\frac{%
1}{2}\phi + \int_\Omega u\otimes u\cdot \nabla\phi +\int_\Omega f_\sigma
\cdot \phi, \quad \hbox{ for all } \phi\in V \hbox{ and } t>0.
\end{equation}
We recall that the set of all weak solutions of \eqref{eq:NS} for all times
will be denoted by $\mathcal{K}$.
\end{definition}

For any $u_0\in H$ and $f_\sigma\in L^2(0,T; V^*)$ there exists at least one
weak solution of \eqref{eq:NS} (see \cite{Temam77}). 
%
%
We recall the following lemma, which is proved exactly as in \cite[Theorem 13]{CKR}.

\begin{lemma}
\label{Converg} For any sequence $u_0^j\to u_0$ weakly in ${H}$, there
exists a sequence $N_j\overset{j\to\infty}{\longrightarrow} \infty$ and a $%
u\in \mathcal{K}$ such that, along a subsequence, for all $T>0$ we have%
\begin{align*}
S_{N_j}(\cdot)u_0^j & \overset{j\to\infty}{\longrightarrow} u(\cdot)\text{
weakly star in }L^{\infty}(0,T;{H}), \\
S_{N_j}(\cdot)u_0^j & \overset{j\to\infty}{\longrightarrow} u(\cdot)\text{
weakly in }L^2 (0,T;H_{\frac12}), \\
S_{N_j}(\cdot)u_0^j & \overset{j\to\infty}{\longrightarrow} u(\cdot)\text{
strongly in }L^2 (0,T;{H}).
\end{align*}
\end{lemma}

We can easily extend this result.

\begin{lemma}
\label{Converg2} The subsequence obtained in Lemma \ref{Converg} also
satisfies 
\begin{equation}  \label{WeakUniform}
S_{N_j}(s_j)u_0^j \overset{j\to\infty}{\longrightarrow} u(t_0) \hbox{ weakly
in } H \hbox{ when } s_j\overset{j\to\infty}{\longrightarrow} t_0, \hbox{
with } \{s_j\}_{j\in \mathbb{N}}\subset[0,\infty).
\end{equation}
\end{lemma}

\begin{proof}
In a standard way (see \cite[p.1494]{KloVal}) we obtain $S_{N_j}(\cdot
)u_0^j\to u(\cdot)$ in $C([0,T],H_{-\frac12})$ for any $T>0$. Combining this
with inequality \eqref{IneqAbsorb}, the statement follows.
\end{proof}

From Lemmas \ref{Converg} and \ref{Converg2} it follows that for any $u_0
\in H$ there exists $u\in\mathcal{KN}$ such that $u(0)=u_0$ (the subclass $%
\mathcal{KN}$ of $\mathcal{K}$ was introduced in Definition \ref{KN}).

\begin{lemma}
\label{Energy} Each $u\in\mathcal{KN}$ satisfies the energy inequality 
\begin{equation}  \label{EnergyIneq}
V(u(t))\leqslant V(u(s))\hbox{ for a.e. } s>0\hbox{ and all } t\geqslant s,
\end{equation}
where 
\[
V(u(r))=\displaystyle\tfrac{1}{2}\|u(r)\|_H^2+\nu\int_0^r\|u(\xi)%
\|_{H_{\frac12}}^2 d\xi-\int_0^r \langle u(\xi),f_{\sigma,-1}\rangle_{\frac12,-\frac12} d\xi.
\]
 Moreover, if $%
u(0)\in{H}\setminus\mathcal{A}$, then \eqref{EnergyIneq} is true for $s=0$
as well.
\end{lemma}

\begin{proof}
Let us first consider $u(0)\in\mathcal{A}$. Then there exist $t_j\overset{%
j\to\infty}{\longrightarrow}\infty$, $\{u_0^j\}_{j\in \mathbb{N}}\subset B_0$
and $N_j\overset{j\to\infty}{\longrightarrow} \infty$ such that 
\begin{equation*}
v_j(t):=S_{N_j}(t+t_j)u_0^j\overset{j\to\infty}{\longrightarrow} u(t)\text{
weakly in }{H},\text{ for all }t\geqslant 0. 
\end{equation*}
The functions $v_j$ satisfy, for all $0\leqslant s\leqslant t$, the energy
equality: 
\begin{equation*}
\frac{1}{2}\| v_j(t)\|_{{H}}^2 +\nu\int_s ^{t}\| v_j(\xi)\|_{{H_{\frac12}}}^2
d\xi-\int_s ^{t}\langle v_j(\xi),f_{\sigma,-1}\rangle_{\frac12,-\frac12}d\xi=\frac{1}{2}\| v_j(s)\|_{{H}}^2 . 
\end{equation*}
By the convergences in Lemmas \ref{Converg}-\ref{Converg2} we have 
\begin{align*}
\| u(t)\|_{{H}}^2 & \leqslant \liminf_{j\to \infty}\ \|v_j(t)\|_H^2 , \\
\int_s^t\| u(\xi)\|_{H_{\frac12}}^2d\xi & \leqslant \liminf_{j\to
\infty}\int_s^t\| v_j(\xi)\|_{{H_{\frac12}^2 }}^2 d\xi, \\
\int_s^t(u(\xi),f_{\sigma,-1})d\xi & = \lim_{j\to
\infty}\int_s^t \langle v_j(\xi),f_{\sigma,-1}\rangle_{\frac12,-\frac12}d\xi \quad \hbox{ for all }
t\geqslant s\geqslant 0, \\
\lim_{j\to \infty}\| v_j(s)\|_H^2 & =\|u(s)\|_H^2 \quad \hbox{ for a.e. }
s>0.
\end{align*}
Hence for a.e. $s>0$ and all $t\geqslant s$, we have 
\begin{equation*}
\frac{1}{2}\| u(t)\|_{{H}}^2 +\nu\int_s ^{t}\| u(\xi)\|_{{H_{\frac12}}}^2
d\xi-\int_s ^{t}\langle u(\xi),f_{\sigma,-1}\rangle_{\frac12,-\frac12}d\xi\leqslant \frac{1}{2}\| u(s)\|_{{H}}^2, 
\end{equation*}
so \eqref{EnergyIneq} follows.

Now let $u(0)\in H\setminus \mathcal{A}$. Then there exists $N_j\overset{%
j\to\infty}{\longrightarrow} \infty$ such that 
\begin{equation*}
v_j(t)=S_{N_j}(t)u(0)\overset{j\to\infty}{\longrightarrow} u(t)\text{ weakly
in }{H},\text{ for all }t\geqslant 0. 
\end{equation*}
Then, arguing as before, we obtain \eqref{EnergyIneq}. Also,
as $v_j(0)=u(0)$ for all $j$, the inequality is true for $s=0$ as well.
\end{proof}

\begin{lemma}
\label{Absorbing2} The ball $B_0$ is absorbing for the solutions in $%
\mathcal{KN}$, that is, for any bounded set $B$ in ${H}$ there exists $%
T_{B}\geqslant0$ such that 
\begin{equation*}
u(t)\in B_0\quad \hbox{ for all }t\geqslant T_{B}, 
\end{equation*}
where $u\in\mathcal{KN}$ and $u(0)\in B$.
\end{lemma}

\begin{proof}
We begin by fixing a bounded subset $B$ of ${H}$. From Lemma \ref{AbsorbingN}
there exists $\tilde{T}=T_{B,B_0}\geqslant0$ such that 
\begin{equation*}
S_{N}(t)u_0\in B_0\quad\hbox{ for all }u_0\in B\cup B_0,\ t\geqslant \tilde{T%
}\hbox{ and }N>0. 
\end{equation*}
Now let $u\in\mathcal{KN}$. If $u(0)\in\mathcal{A}$ and $t\geqslant\tilde{T}$%
, there exists sequences $t_j\overset{j\to\infty}{\longrightarrow}\infty$, $%
\{u_0^j\}_{j\in \mathbb{N}}\subset B_0$ and $N_j\overset{j\to\infty}{%
\longrightarrow} \infty$ such that $S_{N_j}(t+t_j)u_0^j\overset{j\to\infty}{%
\longrightarrow} u(t)$ weakly in $H$. Hence 
\begin{equation*}
\| u(t)\|_{H}\leqslant \liminf_{j\to
\infty}\|S_{N_j}(t+t_j)u_0^j\|_{H}\leqslant1+\frac{\|f_{\sigma,-1}\|_{H_{-\frac12}}^2 }{%
\lambda_1\nu^2}, 
\end{equation*}
since $S_{N_j}(t+t_j)u_0^j\in B_0$ for all $j\in\mathbb{N}$. Hence, $u(t)\in
B_0$.

If $u(0)\in{H}\setminus\mathcal{A}$ and $t\geqslant\tilde{T}$, there exists $%
N_j\overset{j\to\infty}{\longrightarrow} \infty$ such that $S_{N_j}(t)u(0)%
\overset{j\to\infty}{\longrightarrow} u(t)$ weakly in ${H}$. Thus, 
\begin{equation*}
\| u(t)\|_{H}\leqslant\liminf_{j\to \infty}\|S_{N_j}(t)u(0)\|_{H}\leqslant1+\frac{\|f_{\sigma,-1}\|_{H_{-\frac12}}^2 }{%
\lambda_1\nu^2},
\end{equation*}
since $S_{N_j}(t)u(0)\in B_0$ for $t\geqslant\tilde{T}$. Hence, $u(t)\in B_0$
and the proof is complete.
\end{proof}

We observe that the weak topology of ${H}$ is metrizable in the ball $B_0$.
We denote this metric by $\rho_w $. Since all the solutions starting in a
bounded subset $B$ of ${H}$ enter $B_0$ and remain there after some time $%
T_{B}\geqslant0$, we can use this metric to study the asymptotic behavior of
the solutions. For $A,B\subset B_0$ we denote by 
\begin{equation}  \label{def.distw}
dist_w (A,B)=\sup_{a\in A}\inf_{b\in B}\rho_w (a,b)
\end{equation}
the Hausdorff semidistance from $A$ to $B$ in the weak topology.


\bigskip

\begin{proof}[Proof of Theorem \protect\ref{Attractor}.]
\ref{Attractor-a} Let $\{y_j\}\subset\mathcal{A}$. Hence, there exist
sequences $t_j\overset{j\to\infty}{\longrightarrow}\infty$, $\{u_0^j\}_{j\in 
\mathbb{N}}\subset B_0$ and $N_j\overset{j\to\infty}{\longrightarrow} \infty$
such that 
\begin{equation*}
\rho_w (S_{N_j}(t_j)u_0^j,y_j) < \frac1j, 
\end{equation*}
where we can assume that $j$ is sufficiently large so that $%
S_{N_j}(t_j)u_0^j \in B_0$. Hence, up to a subsequence, $%
S_{N_{n}}(t_{n})u_0^{n}\overset{n\to\infty}{\longrightarrow} z$ weakly in $H$
for some $z\in H$ (clearly $z\in \mathcal{A}$). It follows that $y_{n}%
\overset{n\to\infty}{\longrightarrow} z$, and proves that $\mathcal{A}$ is
weakly compact.

\ref{Attractor-b} Let $y\in\mathcal{A}$, $t\geqslant0$, $t_j\overset{%
j\to\infty}{\longrightarrow}\infty$, $\{u_0^j\}_{j\in \mathbb{N}}\subset B_0$
and $N_j\overset{j\to\infty}{\longrightarrow} \infty$ be such that $%
S_{N_j}(t_j)u_0^j\to y$ weakly in $H$. Since $t_j\overset{j\to\infty}{%
\longrightarrow}\infty$, we can assume that $S_{N_j}(t_j-t)u_0^j\in B_0$ for
all $j\in\mathbb{N}$, and hence, up to a subsequence $%
z_0^j:=S_{N_j}(t_j-t)u_0^j\to z$ weakly in $H$ for some $z\in H$. From
Lemmas \ref{Converg} and \ref{Converg2}, there exists $u\in \mathcal{K}$
such that if $[0,\infty)\ni s_j\to t_0$, up to a subsequence, we have 
\begin{equation*}
S_{N_j}(s_j)z_0^j\to u(t_0) \hbox{ weakly in } H. 
\end{equation*}
In particular $u(t)\leftarrow S_{N_j}(t)z_0^j=S_{N_j}(t_j)u_0^j \to y$, and
hence $u(t)=y$. Also, for any $t_0\geqslant0$ and $[0,\infty)\ni s_j\to t_0$
we have 
\begin{equation}  \label{eq:Kn1}
u(t_0) \leftarrow S_{N_j}(s_j)z_0^j = S_{N_j}(t_j-t+s_j)u_0^j.
\end{equation}
When $t_0=0$, we obtain $S_{N_j}(t_j-t)u_0^j\to u(0)$ which proves that $%
u(0)\in\mathcal{A}$. Since $u(0)\in\mathcal{A}$, from \eqref{eq:Kn1} it
follows that $u \in\mathcal{KN}$.

\ref{Attractor-c} Since $u\in\mathcal{KN}$ and $u(0)\in\mathcal{A}$ then
given $t\geqslant0$ we have $S_{N_j}(t+t_j)u_0^j\to u(t)$ weakly in $H$ as $%
j\to\infty$. Hence $u(t)\in\mathcal{A}$, since $t+t_j\overset{j\to\infty}{%
\longrightarrow}\infty$.

\ref{Attractor-d} If that is not the case, there exist $\{u_j\}_{j\in 
\mathbb{N}}\subset \mathcal{KN}$ with $u_j(0)\in B$, $t_j\overset{j\to\infty}%
{\longrightarrow}\infty$ and $\varepsilon>0$ such that 
\begin{equation}  \label{eq:Contrad}
dist_w (u_j(t_j),\mathcal{A})\geqslant\varepsilon \quad \hbox{ for all } j\in%
\mathbb{N}.
\end{equation}
Up to subsequences, we can distinguish two cases: $\{u_j(0)\}_{j\in \mathbb{N%
}}\subset \mathcal{A}$ and $\{u_j(0)\}_{j\in \mathbb{N}}\subset H\setminus%
\mathcal{A}$.

When $\{u_j(0)\}_{j\in \mathbb{N}}\subset\mathcal{A}$, by a diagonalization
process we can construct sequences $r_j\overset{j\to\infty}{\longrightarrow}
\infty$, $\{u_0^j\}_{\mathbb{N}}\subset B_0$ and $N_j\overset{j\to\infty}{%
\longrightarrow} \infty$ such that 
\begin{equation*}
\rho_w (S_{N_j}(r_j+t_j)u_0^j,u_j(t_j))<\frac1j \quad \hbox{ for all } j\in%
\mathbb{N}. 
\end{equation*}
Also, up to a subsequence, we can assume that $S_{N_j}(r_j+t_j)u_0^j \overset%
{j\to\infty}{\longrightarrow} z$ weakly in $H$ for some $z\in H$. Clearly $%
z\in\mathcal{A}$ and $u_j(t_j) \overset{j\to\infty}{\longrightarrow} z$
weakly in $H$. Taking the limit when $j\to\infty$ in \eqref{eq:Contrad} we
obtain a contradiction.

It remains to consider the case $\{u_j(0)\}_{j\in \mathbb{N}}\subset
H\setminus \mathcal{A}$. We can construct a sequence $N_j\overset{j\to\infty}%
{\longrightarrow} \infty$ such that 
\begin{equation*}
\rho_w (S_{N_j}(t_j)u_j(0),u_j(t_j)) < \frac1j. 
\end{equation*}
Again, we can assume that, up to a subsequence, $S_{N_j}(t_j)u_j(0) \overset{%
j\to\infty}{\longrightarrow} z$ weakly in $H$ for some $z\in H$. Since $%
\{u_j(0)\}_{n\in \mathbb{N}}\subset B$, for $j$ sufficiently large, $%
t_j\geqslant T_B$ (where $T_B\geqslant 0$ is given in Lemma \ref{Absorbing2}%
) and $t_j-T_B \overset{j\to\infty}{\longrightarrow} \infty$. Hence $%
S_{N_j}(t_j)u_j(0) = S_{N_j}(t_j-T_B)S_{N_j}(T_B)u_j(0)$ and since $%
S_{N_j}(T_B)u_j(0)\in B_0$, we obtain $z\in\mathcal{A}$. Again, taking the
limit when $j\to\infty$ in \eqref{eq:Contrad} we obtain a contradiction and
prove the result.
\end{proof}

The solutions in $\mathcal{KN}$ satisfy the translation property on $%
\mathcal{A}$.

\begin{lemma}[Translation property]
\label{Translation} For any $u\in\mathcal{KN}$ such that $u(0)\in\mathcal{A}$%
, we have $u_s (\cdot):=u(\cdot+s)\in\mathcal{KN}$ for any $s>0$.
\end{lemma}

\begin{proof}
For each $s\geqslant0$, clearly $u_s \in\mathcal{K}$. From Theorem \ref%
{Attractor}, we know that $u_s (0)\in\mathcal{A}$. If $t_j\overset{j\to\infty%
}{\longrightarrow}\infty$, $\{u_0^j\}_{j\in \mathbb{N}}\subset B_0$ and $N_j%
\overset{j\to\infty}{\longrightarrow} \infty$ are such that 
\begin{equation*}
S_{N_j}(t+t_j)u_0^j \overset{j\to\infty}{\longrightarrow} u(t)\hbox{ weakly
in } H \hbox{ for all } t\geqslant0, 
\end{equation*}
we have 
\begin{equation*}
S_{N_j}(s+t+t_j)u_0^j \overset{j\to\infty}{\longrightarrow} u(t+s)=u_s (t)%
\hbox{ weakly in } H \hbox{ for all } t\geqslant 0, 
\end{equation*}
which shows that $u_s \in\mathcal{KN}$.
\end{proof}

We are not able to prove the translation property for $u\in \mathcal{KN}$
when $u(0)\in H\setminus\mathcal{A}$. The reason is that when we take the
translation of the approximative solutions the initial conditions are a
sequence, and not a constant value equal to the initial value of the limit
problem (as given in the definition of $\mathcal{KN}$). Indeed, if $u\in%
\mathcal{KN}$ and $u(0)\in H\setminus\mathcal{A}$ then, for $s>0$, we define 
$v(\cdot)=u(\cdot+s)$. We choose $s>0$ such that $u(s)\in H\setminus\mathcal{%
A}$. For $t_0\geqslant s$ and $[s,\infty)\ni s_j\to t_0$ we have%
\begin{equation*}
S_{N_j}(s_j)u(0)=S_{N_j}(s_j-s)S_{N_j}(s)u(0)=S_{N_j}(\widetilde{s}_j)u_0^j 
\overset{j\to\infty}{\longrightarrow} u(t_0)=v(t_0-s)=v(\widetilde{t}_0), 
\end{equation*}
where $\widetilde{s}_j=s_j-s$, $u_0^j=S_{N_j}(s)u(0),\ \widetilde{t}_0=t_0-s$%
. However, as remarked before, the initial conditions are not equal to $u(s)$%
, so we do not have the convergence 
\begin{equation*}
S_{N_j}(\widetilde{s}_j)u(s) \overset{j\to\infty}{\longrightarrow} v(%
\widetilde{t}_0). 
\end{equation*}

As a consequence of this property the map $G\colon\mathbb{R}^{+}\times%
\mathcal{A}\to P(\mathcal{A})$, where $P(\mathcal{A})$ is the set of all
nonempty subsets of $\mathcal{A}$, given by 
\begin{equation*}
G(t,u_0)=\{u(t)\colon u\in\mathcal{KN}\hbox{ and } u(0)=u_0 \in \mathcal{A}%
\} 
\end{equation*}
is a multivalued semiflow, that is, $G(0,\cdot)$ is the identity map and $%
G(t+s,u_0)\subset G(t,G(s,u_0))$ for all $t,s\geqslant0$ and $u_0 \in%
\mathcal{A}$, as shown in the next result.

\begin{proposition}
$G$ is a multivalued semiflow.
\end{proposition}

\begin{proof}
Fix $u_0\in\mathcal{A}$. If $y\in G(t+s,u_0)$ then there exists $u\in%
\mathcal{KN}$ with $u(0)=u_0$ such that $y=u(t+s)$. We have $y=u_s (t)$ and
from the translation property we obtain $u_s \in \mathcal{KN}$ and $u_s
(0)=u(s)$. Hence $y\in G(t,u(s))\subset G(t,G(s,u_0))$. Therefore, $%
G(t+s,u_0)\subset G(t,G(s,u_0))$.
\end{proof}

We have not been able to prove that $G$ is a strict multivalued semiflow,
that is, that the equality $G(t+s,u_0)=G(t,G(s,u_0))$ is satisfied for all $%
t,s\geqslant0$ and $u_0 \in\mathcal{A}$. The reason is that we do not know
whether the concatenation of solutions from $\mathcal{KN}$ remains in $%
\mathcal{KN}$ (even when $u_0 \in\mathcal{A}$), that is, we do not know that
if given $u_1,u_2\in\mathcal{KN}$ such that $u_2(0)=u_1(s)$ for some $s>0$,
the function 
\begin{equation*}
u(t)=\left\{%
\begin{array}{cl}
u_1(t) & \hbox{ for } 0 \leqslant t \leqslant s, \\ 
u_2(t-s) & \hbox{ for } t>s,%
\end{array}
\right. 
\end{equation*}
belongs to $\mathcal{KN}$.

\begin{lemma}
\label{k4} Let $u_0^n\to u_0$ weakly in $H$, where $\{u_0^n\}_{n\in \mathbb{N%
}} \subset \mathcal{A}$ and $u_0\in\mathcal{A}$. For any sequence $u_n \in%
\mathcal{KN}$ with $u_n(0)=u_0^n$ there exists a subsequence (renamed the
same) and $u\in\mathcal{KN}$ with $u(0)=u_0$ such that 
\begin{equation*}
u_n(s_n) \overset{n\to\infty}{\longrightarrow} u(t) \hbox{ weakly in }H %
\hbox{ for all } \lbrack 0,\infty)\ni s_n \overset{n\to\infty}{%
\longrightarrow} t. 
\end{equation*}
\end{lemma}

\begin{proof}
For each $n\in \mathbb{N}$ there exist sequences $t_j^n\overset{j\to\infty}{%
\longrightarrow} \infty$, $\{u_0^{j,n}\}_{j\in \mathbb{N}}\subset B_0$ and $%
N_j^{n}\overset{j\to\infty}{\longrightarrow} \infty$ such that for any given 
$T>0$ we have 
\begin{equation*}
\rho_w (S_{N_j^{n}}(t+t_j^{n})u_0^{j,n},u_{n}(t))\xrightarrow{j\to \infty} 0
\quad \hbox{ uniformly in }\lbrack 0,T]. 
\end{equation*}
Thus, for each $n\in \mathbb{N}$ we can choose $j_{n}>n$ such that%
\begin{equation*}
\rho_w (S_{N_j^{n}}(t+t_{j_{n}}^{n})u_0^{j_{n},n},u_{n}(t))<\frac{1}{n}\quad 
\text{ for all } t\in[0,T]. 
\end{equation*}
Clearly, we can choose the sequence $\{j_n\}_{n\in \mathbb{N}}$ in such a
way that it is strictly increasing. By Lemmas \ref{Converg} and \ref%
{Converg2} there exists $u\in \mathcal{K}$ such that, up to a subsequence,
we have%
\begin{equation*}
\rho_w (S_{N_j^{n}}(t+t_{j_{n}}^{n})u_0^{j_{n},n},u(t))\xrightarrow{n\to
\infty} 0\quad \text{ uniformly in } [0,T]. 
\end{equation*}
Hence, $u\in\mathcal{KN}$ and, if $[0,T]\ni s_{n} \overset{n\to\infty}{%
\longrightarrow} t$, for any $\varepsilon>0$ there exists $%
n_0=n_0(\varepsilon)\in \mathbb{N}$ such that for $n\geqslant n_0$ we have 
\begin{align*}
\rho_w& (u_{n}(s_{n}),u(t)) \\
& \leqslant \rho_w
(S_{N_j^{n}}(s_{n}+t_{j_{n}}^{n})u_0^{j_{n},n},u(t))+\rho_w
(S_{N_j^{n}}(s_{n}+t_{j_{n}}^{n})u_0^{j_{n},n},u_{n}(s_{n})) \leqslant
\varepsilon+\frac{1}{n},
\end{align*}
so $u_{n}(s_{n}) \overset{n\to\infty}{\longrightarrow} u(t)$ weakly in $H$.
Clearly $u(0)=u_0$ and, by a diagonalization argument, the conclusion holds
for any $t\geqslant 0.$
\end{proof}

We obtain now a standard characterization of the weak global attractor $%
\mathcal{A}$. The function $\phi\colon \mathbb{R}\to H$ is said to be a 
\textbf{complete trajectory} of \eqref{eq:NS} in $\mathcal{KN}$ if 
\begin{equation*}
\phi(\cdot+s)\! \mid_{[0,\infty)}\ \in\mathcal{KN} \quad \hbox{ for any }
s\in\mathbb{R}. 
\end{equation*}
It is said to be a \textbf{bounded complete trajectory} in $\mathcal{KN}$ if
it is a complete trajectory in $\mathcal{KN}$ and $\phi(\mathbb{R})$ is a
bounded set in $H$.

\begin{lemma}
\label{Charac1} If $\phi$ is a bounded complete trajectory in $\mathcal{KN}$%
, then $\phi(\mathbb{R})\subset\mathcal{A}$.
\end{lemma}

\begin{proof}
Denote $B=\phi(\mathbb{R})$ and fix $t_0\in\mathbb{R}$. Denote $%
u_{\tau}(\cdot)=\phi(\cdot+t_0-\tau)$ for $\tau\geqslant 0$. Then, as $%
\phi(t_0)=u_{\tau}(\tau)$, we deduce by the attracting property of $\mathcal{%
A}$ that 
\begin{align*}
dist_w(\phi(t_0),\mathcal{A}) & = dist_w (u_{\tau}(\tau),\mathcal{A}) \\
& \leqslant \sup_{u\in\mathcal{KN},\ u(0)\in B}dist_w (u(\tau),\mathcal{A}) 
\overset{\tau\to\infty}{\longrightarrow} 0.
\end{align*}
Therefore, $\phi(t_0)\in\mathcal{A}$. Since $t_0\in \mathbb{R}$ is
arbitrary, the result follows.
\end{proof}

\begin{theorem}
\label{charac2} The weak global attractor $\mathcal{A}$ can be characterized
by the union of all bounded complete trajectories of \eqref{eq:NS} in $%
\mathcal{KN}$, that is, 
\begin{equation*}
\mathcal{A}=\{\phi(0)\colon\phi\hbox{ is a bounded complete trajectory of
\eqref{eq:NS} in }\mathcal{KN}\}. 
\end{equation*}
\end{theorem}

\begin{proof}
{Let }$z\in\mathcal{A}$ be arbitrary. We take a sequence $t_{n}\overset{%
n\to\infty}{\longrightarrow} \infty$. Since $\mathcal{A}$ is
negatively invariant, there exist $u_{n}\in\mathcal{KN}$ such that $%
u_{n}(t_{n})=z$ and $u^{n}(0)\in\mathcal{A}$. We define the functions $%
v_{n}^{0}(\cdot)=u_{n}(\cdot+t_{n})$, which belongs to $\mathcal{KN}$ by the
translation property. As the attractor is positively invariant, $%
v_{n}^{0}(t)\in\mathcal{A}$ for all $t\geqslant 0$. In view of Lemma \ref{k4}%
, up to a subsequence, $v_{n}^{0}(s_{n}) \overset{n\to\infty}{\longrightarrow%
} v^{0}(t)\in\mathcal{A}$ weakly in $H$ for all $t\geqslant 0$ and $%
[0,\infty)\ni s_{n} \overset{n\to\infty}{\longrightarrow} t$, where $%
v^{0}(\cdot)\in\mathcal{KN}$. In addition, $v^{0}(0)=z$. Consider now the
sequence $v_{n}^{1}(\cdot )=u_{n} (\cdot +t_{n}-1)$, which again belongs to $%
\mathcal{KN}$ and satisfies $v_{n}^{1}(t)\in\mathcal{A}$ for all $t\geq0$.
Using again Lemma \ref{k4}, up to a subsequence we have $v_{n}^{1}(s_{n}) 
\overset{n\to\infty}{\longrightarrow} v^{1}(t)\in \mathcal{A}$ weakly in $H$
for all $t\geqslant 0$ and $[0,\infty)\ni s_{n} \overset{n\to\infty}{%
\longrightarrow} t$, where $v^{1}(\cdot )\in\mathcal{KN}$. It follows that $%
v^{1}(t+1)=v^{0}(t)$ for all $t\geqslant 0$ as well. In this way, we
construct inductively a sequence of solutions from $\mathcal{KN}$, denoted
by $\{v^j(\cdot )\}_{j\in \mathbb{N}}$, such that $v^j(t)\in$ $\mathcal{A}$
and $v^j(t+1)=v^{j-1}(t)$ for all $t\geqslant 0$ and $j$. Define $\phi\colon 
\mathbb{R} \to H$ to be at any $t$ the common value of the functions $v^j$,
that is, 
\begin{equation*}
\phi(t)=v^j(t+j)\text{ if }t\geqslant -j. 
\end{equation*}
Then $\phi$ is a bounded complete trajectory of \eqref{eq:NS} in $\mathcal{KN%
}$. Thus, 
\begin{equation*}
\mathcal{A}\subset\{\phi(0)\colon\phi\hbox{ is a bounded complete trajectory
of \eqref{eq:NS} in }\mathcal{KN} 
.\} 
\end{equation*}
The converse inclusion follows from Lemma \ref{Charac1}, and the result is
proved.
\end{proof}

Finally, we will prove that the global attractors $\mathcal{A}_N$ of %
\eqref{VeryWeakNSE} behave upper semicontinuously with respect to the weak
global attractor $\mathcal{A}$ of \eqref{eq:NS}.

\begin{proposition}
\label{prop.Cont} If $\mathcal{A}_{N}$ is the global attractor of $S_{N}$
then 
\begin{equation*}
dist_w (\mathcal{A}_{N},\mathcal{A}) \overset{N\to\infty}{\longrightarrow}
0. 
\end{equation*}
\end{proposition}

\begin{proof}
{If $N_j\overset{j\to\infty}{\longrightarrow} \infty$ and $u_j\in\mathcal{A}%
_{N_j}$ for each $j\in\mathbb{N}$, since $\cup_{N>0}\mathcal{A}_{N}\subset
B_0$, there exists $u_0\in{H}$ such that, up to a subsequence, $u_j\to u_0$
weakly in ${H}$. For each $u_j$ there exists a global solution $\xi_j$ of
the semigroup $S_{N_j}$ with $\xi(0)=u_j$ and $\xi (\mathbb{R})\subset%
\mathcal{A}_{N_j}$. Hence 
\begin{equation*}
S_{N_j}(j)\xi_j(-j)=\xi_j(0)=u_j \overset{j\to\infty}{\longrightarrow} u_0%
\hbox{ weakly in }{H}. 
\end{equation*}
This shows that $u_0\in\mathcal{A}$. Hence, $dist_w (\mathcal{A}_{N},%
\mathcal{A}) \overset{N\to\infty}{\longrightarrow} 0$. }
\end{proof}


\end{document}